\documentclass[12pt, leqno]{article}
\usepackage{amsmath}
\usepackage{amsthm}
\usepackage{amssymb}
\usepackage{latexsym}
\usepackage{graphicx}
\allowdisplaybreaks[1]
\usepackage{amsfonts}
\usepackage{ascmac}
\usepackage{color}
\usepackage{enumerate}
\usepackage{ulem}

\theoremstyle{definition}
\theoremstyle{plain}
\allowdisplaybreaks[1]
\date{}
\renewcommand{\theequation}{\arabic{section}.\arabic{equation}}
\newtheorem{Thm}{Theorem}[section]
\newtheorem{Prop}[Thm]{Proposition}
\newtheorem{Lemma}[Thm]{Lemma}
\newtheorem{Cor}[Thm]{Corollary}

\newcommand{\p}{\partial}

\newcommand{\dis}{\displaystyle}

\newcommand{\N}{{\mathbb N}}
\newcommand{\R}{{\mathbb R}}

\newcommand{\ep}{\varepsilon }
\newcommand{\2}{\frac{1}{2} }

\newcommand{\tra}{{\sf T}}
\newcommand{\interface}{\Sigma}

\def\text#1{\mbox{#1 }}

\title{\bf Mathematical analysis of the \\velocity extension level set method}
\author{Dieter Bothe\footnote{Department of Mathematics, Technische Universit\"at Darmstadt, Germany. E-mail:  bothe@mma.tu-darmstadt.de}
 \,\,and \,\, Kohei Soga\footnote{Department of Mathematics, Faculty of Science and Technology, Keio University, Japan. E-mail:  soga@math.keio.ac.jp
}}
\begin{document}
\maketitle
\begin{abstract}
\noindent  A passively advected sharp interface can be represented as the  zero level set  of a  level set  function $f$. 
The linear transport equation $\p_tf+v\cdot \nabla f =0$ is the simplest governing equation for such a  level set  function.  
While the signed distance of the interface is a geometrically convenient function,  e.g., the norm of the gradient is everywhere one,  its time evolution is not governed by the linear transport equation. 
In computational fluid dynamics, several  modifications of  the simplest case have been proposed in order to compute  the signed distance function or to stabilize the norm of the gradient of a  level set  function on the interface. 
The velocity extension method is a prominent   method used for efficient numerical approximation  of the local signed distance function of the interface. 
Our current paper presents a rigorous mathematical formulation of the velocity extension method and  proves that the method provides the local signed distance function of the moving interface.           
A key is to derive a first-order fully nonlinear PDE that is equivalent to the linear  transport equation with  extended velocity. 
Then, wellposedness of the PDE is established in the class of $C^2$-smooth solutions, global in time and local in space, with the local signed distance property, as well as  in the class of $C^0$-viscosity solutions, global in time and space.   
Furthermore, partial regularity of the viscosity solution is proven, thus  confirming that, if initial data is smooth near the initial interface, the viscosity solution is smooth in a time-global tubular neighborhood of the interface, coinciding with the local-in-space $C^2$-smooth solution.     

\medskip

  \medskip

  
\end{abstract}
\setcounter{section}{0}
\setcounter{equation}{0}
\section{Introduction}

Let $v=v(t,x):[0,\infty)\times\Omega\to\R^3$ be an (Eulerian) velocity field of a fluid flow in a bounded domain $\Omega\subset\R^3$. 
Although investigation of governing equations for $v$, e.g., the Navier-Stokes system,  is one of the central issues in fluid dynamics, we do not look at such a problem; instead we focus on advection of a scalar quantity with respect to $v$, assuming that $v$ is given (hypothesis on $v$ are stated in Section 2).    
A passively advected (Eulerian) scalar quantity $f=f(t,x):[0,\infty)\times\Omega\to\R^3$ is governed by the linear transport equation
\begin{align}\label{LT}
\frac{\p f}{\p t}(t,x)+v(t,x)\cdot\nabla f(t,x)=0,\quad f(0,\cdot)=\phi^0, 
\end{align} 
where $\nabla=(\p_{x_1},\p_{x_2},\p_{x_3})$. We refer to the introduction in \cite{BFS} for more details on the linear transport equation in fluid dynamics. 
The linear transport equation is closely related to the flow map $X=X(s,t,\xi)$ of the ordinary differential equation (ODE)  
\begin{align}\label{ODE}
x'(s)=v(s,x(s)),
\end{align}
where the (unique) solution of \eqref{ODE} to the  initial condition $x(\tau)=\xi$ is then denoted by $X(s,\tau,\xi)$.   
The (classical) solution of  \eqref{LT} is represented as 
\begin{align}\label{LT-sol} 
f(t,x)=\phi^0(X(0,t,x)). 
\end{align}
This paper focuses on the  level set  method for two-phase flows with a sharp interface, where the moving interface can be described by the  zero level set  of a  level set  function governed by \eqref{LT}. 
Let $\Omega^+(0)$ be a subdomain of $\Omega$ such that $\interface(0):=\p \Omega^+(0)$ is a closed $C^2$-smooth surface that does not intersect $\p\Omega$. Set $\Omega^-(0):=\Omega\setminus(\Omega^+(0)\cup\interface(0))$.  
Suppose that $\Omega^\pm(0)$ are filled with two immiscible fluids distinguished with upper index $\pm$.   
For each $t\ge0$, let $\Omega^\pm(t)\subset \Omega$ denote the set occupied by fluid$^\pm$. We assume that 
$$\mbox{$\Omega^{\pm}(t)=X(t,0,\Omega^\pm(0))$\quad\mbox{ \,\,\, ($\forall\,t\,\ge0$)}.}$$ 
This means that the trajectories $x(\cdot)$  governed by \eqref{ODE} of arbitrary fluid elements in $\Omega^\pm(0)$ can never leave $\Omega^{\pm}(t)$ for all $t\ge0$, i.e., no phase change occurs. 
We further assume that $\p\Omega^+(t)$ stays  away from $\p\Omega$ for all $t$. Then the interface $ \interface(t):=\p\Omega^+(t)=\p\Omega^-(t)\setminus\p\Omega$ is described as 
 $$\interface(t)=X(t,0,\interface(0))\quad\mbox{ \,\,\, ($\forall\,t\,\ge0$)},$$
 which is called a material interface.  
In this situation, the interface is passively advected by the velocity field $v$ without relative motion in the normal direction. 
Thus the speed of normal  (toward $\Omega^-(t)$)  displacement  $V_\interface$ of the moving interface satisfies $V_\interface=v\cdot \nu$ ($\nu$ stands for the normal vector pointing toward $\Omega^-(t)$), which constitutes the so-called kinematic boundary condition.  We remark that $\interface(t)$ is a closed smooth surface for each $t\ge0$, since $X(t,0,\cdot)$ is a diffeomorphism.

In view of the solution formula \eqref{LT-sol} for \eqref{LT}, the interface $\interface(t)$ can be represented as the  zero level set  of a solution $f$ of \eqref{LT}: 
\begin{align*} 
\interface(t)=\{x\in\Omega\,|\,f(t,x)=0\},
\end{align*}
where the initial data $\phi^0:\Omega\to\R$ is given so that  
\begin{align}\label{2initial}
\left. 
\begin{array}{lll}\medskip 
&&\mbox{ $\phi^0>0$ on $\Omega^+(0)$,  \qquad\qquad \,\,\,\,\,\,\quad$\phi^0<0$ on $\Omega^-(0)$}, \\
&&\mbox{$\{x\in\Omega\,|\,\phi^0(x)=0\}=\interface(0)$,}\quad \mbox{ $\nabla \phi^0\neq 0$ on $\interface(0)$.} 
\end{array}
\right.
\end{align}
Moreover, it holds that 
\begin{align*} 
\Omega^+(t)=\{x\in\Omega\,|\,f(t,x)>0\},\quad \Omega^-(t)=\{x\in\Omega\,|\,f(t,x)<0\}. 
\end{align*}
An immediate consequence of the  level set  representation of $\interface(t)$ is that the unit normal vector field $\nu$  and the local curvature  $\kappa$ (with respect to $\nu$) of $\interface(t)$ are given as   
\begin{align}\label{mean}
 \nu(t,x)=-\frac{\nabla f(t,x)}{|\nabla f(t,x)|}\mbox{ (pointing toward $\Omega^-(t)$)},\quad
 \kappa(t,x)= -\nabla\cdot \nu(t,x). 
\end{align} 
More precisely, $\kappa(t,x)$ denotes twice the mean curvature of $\interface(t)$ at each $x\in\interface(t)$ and is represented as $-\nabla_{\interface(t)}\cdot\nu(t,x)$. Since $\nu(t,\cdot)$ from \eqref{mean} is now given in a neighborhood of $\interface(t)$ with $|\nu(t,\cdot)|\equiv 1$ there, it holds that $ \kappa(t,x)= -\nabla\cdot \nu(t,x)$ as stated. Furthermore, it also holds that $\kappa(t,x)=\Delta f(t,x)$ if $|\nabla f(t,\cdot)|\equiv 1$ in a neighborhood of  $x\in\interface(t)$.   
Note that the Navier-Stokes system for $v$ and \eqref{LT} are coupled through $\nu$ and $\kappa$ in a two-phase flow with a capillary interface, i.e., with surface tension acting on $\interface(t)$ (see, e.g., \cite{pruss} for details).  

In computational fluid dynamics, the  level set  method is widely applied to two-phase flow problems, where $\nu$ and $\kappa$ are challenging objects from numerical viewpoints. The major difficulty comes from the fact that passive advection allows $|\nabla f(t,\cdot)|$ to become (exponentially) very small or very large as $t$ increases; this would destroy accuracy of numerical construction of $\nu$ and $\kappa$. In fact, the evolution of $|\nabla f(t,\cdot)|$ along each characteristic curve $x(\cdot)$ (a solution of \eqref{ODE}) is governed by 
\begin{align}\label{evo1}
&\frac{1}{2}\frac{\rm d}{{\rm d}s} |\nabla f(s,x(s))|^2=-\langle \nabla v(s,x(s)) \nabla f(s,x(s)),\nabla f(s,x(s))\rangle, \\\nonumber
& |\nabla f(0,x(0))|=|\nabla \phi^0(x(0))|,
\end{align}  
where $|x|$, $\langle x,y\rangle=x \cdot y$ stand for the Euclidian norm, inner product of $\R^3$, respectively. 
Many efforts have been reported in the literature in order to stabilize  $|\nabla f(t,x)|$ on/near the interface. We refer to three types of stabilization methods: {\it reinitialization methods, nonlinear modification methods} and  {\it the velocity extension method}.  

Reinitialization methods, pioneered in \cite{Sussman1994} and \cite{Sussman1999}, deal with the local signed distance function of the interface. If $\phi^0$ locally coincides with the local signed distance function of $\interface(0)$, it holds that $ |\nabla f(0,\cdot)|\equiv1$ in a neighborhood of $\interface(0)$.
However, \eqref{LT} does not transport the signed distance properties, i.e., $|\nabla f(t,\cdot)|\not\equiv 1$ for $t>0$ near $\interface(t)$ in general. 
The strategy of reinitialization methods is the following: 
{\it
\begin{enumerate}
 \item Consider $0=t_0<t_1<t_2<\cdots$ and solve the linear transport equation \eqref{LT} within each  $[t_i,t_{i+1}]$ with initial condition $f(t_i,\cdot)=[${\rm the local signed distance function of }$\interface(t_i)]$.
 \item  Calculate the local signed distance function of $\interface(t_{i+1})$.
 \item Reset $f(t_{i+1},\cdot)$ to be the local signed distance function of $\interface(t_{i+1})$.
 \end{enumerate}} 
\noindent By this procedure,  $|\nabla f(t,\cdot)|$ is expected to stay close to $1$ near $\interface(t)$  for all $t\ge0$, or $f$  is expected to stay close to the local signed distance function of the interface for all $t\ge0$.    
In \cite{Sussman1994}, the step 2 was numerically done through long time behavior in the corrector equation 
\begin{align}\label{1corr}
\frac{\p \psi}{\p s}(s,x)={\rm sign}(\psi(s,x))(1-|\nabla\psi(s,x)|), \quad \psi(0,\cdot)=f(t_{i+1},\cdot)\qquad ({\rm sign}(0):=0).
\end{align}
There, it was observed (formally) that $\psi(s,\cdot)$ tends to a solution $\bar{\psi}$ of the eikonal equation $|\nabla\psi|=1$ as $s\to\infty$, and hence $\bar{\psi}$ is the signed distance function of $\interface(t_{i+1})$. 
In \cite{HN}, the authors demonstrated rigorous mathematical analysis of reinitialization methods for first-order Hamilton-Jacobi equations in the whole space $\R^n$ in the class of viscosity solutions in the following way: {\it first, describe the switching process of solving a evolutional Hamilton-Jacobi equation for the level set function (corresponding to the step 1) and the corrector equation \eqref{1corr} (corresponding to  the step 2) in terms of a single Hamilton-Jacobi equation generated by $t$-periodic/$t$-discontinuous Hamiltonian with several scaling parameters; then, consider the limit of the period to be $0$ (corresponding to  $t_{i+1}-t_i\to0$) and formulate the limit problem through a  homogenization technique; finally take the limit in the homogenized problem with respect to a scaling parameter (corresponding to $s\to\infty$ in  \eqref{1corr}), where the viscosity solution of the homogenized problem converges to the signed distance function of the original interface, even with initial data that is not the corresponding signed distance function of the initial interface.} 
It is interesting to note that the scaling limit in the  homogenized problem creates the signed distance property, not only  preserving it.

Nonlinear modification methods, starting with \cite {Roisman}, add a nonlinear term to the transport equation \eqref{LT} such that the advection by $v$ and control of the norm of the gradient on the interface take place simultaneously while solving a  single PDE.
Of course, the interface must be unchanged by any such modification.  
The well-known example of such a nonlinear modification is the PDE 
\begin{align}\label{mod-PDE1}
 \frac{\p \phi}{\p t}(t,x) +v(t,x)\cdot \nabla \phi(t,x)=\phi(t,x) \Big<\nabla v(t,x)  \frac{\nabla \phi(t,x)}{|\nabla \phi(t,x)|}, \frac{\nabla \phi(t,x)}{|\nabla \phi(t,x)|}\Big>,
\end{align}
 which is designed so that $ |\nabla \phi(\cdot,x(\cdot))|$ evolves along each characteristic curve $x(\cdot)$ (this is no longer a solution of \eqref{ODE})  located on the interface  according to 
\begin{align*}
&\frac{1}{2}\frac{\rm d}{{\rm d}s} |\nabla \phi(s,x(s))|^2=0.
\end{align*} 
If PDE \eqref{mod-PDE1} is solved with the initial data given to \eqref{LT},   
 formal calculation by the method of characteristics implies  that the nonlinear term in \eqref{mod-PDE1} does not change  the  interface (i.e., the  zero level set  of $\phi(t,\cdot)$ is $\interface(t)$) and preserves  $|\nabla\phi|$ on the interface in the sense that
 \begin{align}\label{1111pre}
 \forall\,t\in[0, \infty) \,\,\,\forall\,x\in\interface(t)\quad \exists\,\xi\in \interface(0) \mbox{ \,\,\, such that\,\,\,\,\,\,}  |\nabla \phi(t,x)|=|\nabla\phi^0(\xi)|.     
\end{align}      
However, since  \eqref{mod-PDE1} is fully nonlinear, wellposedness and regularity of solutions are not at all trivial, where $(t,x)$-global classical solutions cannot be expected in general.    
In \cite{F}, the authors numerically investigated  \eqref{mod-PDE1}. 
In \cite{H}, the author formulated essentially the same kind of nonlinear modification to first-order Hamilton-Jacobi equations in the whole space and demonstrated rigorous analysis in the class of viscosity solutions: {\it if initial data locally coincides with the signed distance function  of the initial interface, a $C^0$-viscosity solution $u(t,x)$ exists and it is differentiable with respect to $x$ on the interface, satisfying $|\nabla u(t,x)|=1$. Furthermore, $u$ stays close to the signed distance function $d$ (assumed to be smooth) of the interface in the sense that for every $\ep>0$ there exists a positive constant $\rho(\ep)>0$ such that} 
\begin{align*}
&e^{-\ep t}d(t,x)\le u(t,x)\le e^{\ep t}d(t,x) \quad \mbox{if $0\le d(t,x)\le \rho(\ep)$},\\
&e^{\ep t}d(t,x)\le u(t,x)\le e^{-\ep t}d(t,x) \quad \mbox{if $- \rho(\ep)\le d(t,x)\le 0$}. 
\end{align*}     
In \cite{BFS}, the authors investigated \eqref{mod-PDE1} in a bounded domain: {\it \eqref{mod-PDE1} is solvable, yielding the same  zero level set  on a time-global tubular neighborhood of the interface (see Section 2 below for definition) in the class of $C^2$-solutions with the preservation of $|\nabla \phi|$ on the interface in the sense of \eqref{1111pre};  \eqref{mod-PDE1} with a cut-off to remove the singularity $\frac{p}{|p|}$ at $p=0$ and the nonlinearity near $\p\Omega$ is $(t,x)$-globally solvable in the class of viscosity solutions; the viscosity solution coincides with the  $C^2$-solution  on a time-global tubular neighborhood of the interface.}     
The last assertion means that the formulas in \eqref{mean} make sense with the $C^0$-viscosity solution. 
We remark that, in a practical situation, e.g., numerical construction of solutions, it would be convenient to investigate the PDE on the whole domain $\Omega$.  
We emphasize that the nonlinear modification method based on \eqref{mod-PDE1} allows such strong control of $|\nabla \phi|$ {\it on} the interface, but it does not provide the local signed distance function of the interface.       
   
The following nonlinear modification was also  investigated in \cite{BFS}: 
$$\frac{\p\phi}{\p t}(t,x)+v(t,x)\cdot\nabla\phi(t,x)=\phi(t,x)(\beta-|\nabla\phi(t,x)|). $$
This provides weaker control of  $|\nabla \phi|$  within a neighborhood of  the interface; namely, 
a suitable choice of the constant $\beta$ (depending on $|\nabla v|$ and $|\nabla\phi^0|$) guarantees an a priori bound of $|\nabla \phi|$ within a small neighborhood of  the interface;  
however,  $|\nabla \phi|$ is not necessarily preserved even on the interface.

 The velocity extension method, pioneered in \cite{ZCMO},  \cite{JS1}, \cite{JS2} and \cite{AJS}, introduces a new velocity field  $\tilde{v}$ in a sufficiently small neighborhood of the interface  $\interface(t)$ via constant extension of $v(t,\cdot)|_{\interface(t)}$ along the normal direction for each $t\ge 0$.   
Having $\tilde{v}$, the new linear transport equation 
 \begin{align}\label{Set1}
\frac{\p \phi}{\p t}(t,x)+\tilde{v}(t,x)\cdot\nabla\phi(t,x)=0
\end{align} 
is solved.
In rigorous mathematical terms, the new velocity field $\tilde{v}$ is defined  in a small neighborhood of $\interface(t)$ as  
$$\tilde{v}(t,x):=v(t,\mathcal{P}^tx),$$
where $\mathcal{P}^tx$ is the metric projection onto $\interface(t)$ (see Section 2 for more precise descriptions). 
The advantage of using $\tilde{v}$ instead of $v$ is described in the following statement \cite{ZCMO},  \cite{JS1}, \cite{JS2}, \cite{AJS}, \cite{Ober}: 
\medskip\medskip

\noindent{\bf Claim 1}\,(outcome of  the velocity extension method){\bf.}  {\it  The PDE \eqref{Set1} with initial data equal to the local signed distance  function of $\interface(0)$ provides a $C^2$-solution $\phi$ such that its zero level set is equal to $\interface(t)$. Furthermore, it holds that $|\nabla \phi(t,\cdot)|\equiv1$ not only on $\interface(t)$ but also in a small neighborhood of  $\interface(t)$ for all $t\ge0$. Thus $\phi(t,\cdot)$ is the $C^2$-smooth local signed distance  function of $\interface(t)$ for all $t\ge0$.}

\medskip\medskip

\noindent More detailed statements will be given in Section 2.
 To the best of the authors' knowledge, rigorous mathematical analysis on  the velocity extension method is not present in the literature. 
In fact,  Claim 1  seems highly nontrivial as \eqref{Set1} is a nonlinear problem due to the metric projection $\mathcal{P}^tx$. 
The purpose of the present paper is to provide a rigorous formulation of the  velocity extension method and to prove Claim 1, also making clear the connection with the  above mentioned nonlinear modification methods. 

In Section 2, we describe difficulties dealing directly with the   ``linear'' equation \eqref{Set1} to prove Claim 1. 
Our strategy to overcome the difficulty is to derive the following fully nonlinear PDE that turns out to be equivalent to \eqref{Set1}:   
\begin{align}\label{mod123}
  \frac{\p \phi}{\p t}(t,x)+v\Big(t,x-\phi(t,x)\nabla\phi(t,x)\Big)\cdot\nabla\phi(t,x)=0.      
\end{align}      
We remark that the equation \eqref{mod123} was  formally derived in \cite{ZCMO}. 
Then, with initial data equal to the local signed distance function of $\interface(0)$,  we will prove that there exists a time-global $C^2$-solution $\phi$ defined on a tubular neighborhood of the interface, where the solution possesses the desired properties; in fact, $\phi(t,\cdot)$ is the local signed distance function of $\interface(t)$ for all $t\ge0$.         
In other words, we rigorously show that \eqref{mod123} is the evolution equation for the local signed distance function of $\interface(t)$, possessing global-in-time/local-in-space $C^2$-solutions.  
Furthermore, we consider the PDE 
\begin{align}\label{mod1234}
  \frac{\p \phi}{\p t}(t,x)+v\Big(t,x-\frac{\phi(t,x)\nabla\phi(t,x)}{|\nabla\phi(t,x)|^2}\Big)\cdot\nabla\phi(t,x)=0,      
\end{align} 
which is a generalization of  \eqref{mod123}. 
The PDE \eqref{mod1234} meets our requirements with  initial data being not necessarily the local signed distance function. 
We will prove that \eqref{mod1234}, with initial data satisfying only \eqref{2initial}, provides a time-global $C^2$-solution $\phi$ defined on a tubular neighborhood of the interface, where {\it $|\nabla \phi|$ is preserved not only on the interface but also in a neighborhood  the interface}.  
We remark that Taylor's approximation for $v\Big(t,x-\frac{\phi(t,x)\nabla\phi(t,x)}{|\nabla\phi(t,x)|^2}\Big)$ with respect to $-\frac{\phi(t,x)\nabla\phi(t,x)}{|\nabla\phi(t,x)|^2}$  in \eqref{mod1234} yields  \eqref{mod-PDE1}, where we stress again that \eqref{mod-PDE1} preserves $|\nabla \phi|$ only on the interface.  Hence, one can see that \eqref{mod1234} provides stronger control of $|\nabla \phi|$  than nonlinear modification methods.   

Finally,  we will establish $(t,x)$-global wellposedness of \eqref{mod1234} with regularization of the singularity from  $\frac{p}{|p|^2}$ in the class of $C^0$-viscosity solutions and  partial $C^2$-regularity of the viscosity solution in a time-global tubular neighborhood of  the interface, where the viscosity solution locally coincides with the global-in-time/local-in-space  $C^2$-solution of \eqref{mod1234}.  Thus,  a $C^0$-viscosity solution of  \eqref{mod1234} provide the $C^2$-smooth local signed distance function of $\{\interface(t)\}_{t\ge0}$. 
\setcounter{section}{1}
\setcounter{equation}{0}
\section{Main results}
We now give rigorous statements on the  velocity extension method.   
Let $\Omega\subset\R^3$ be a bounded connected open set and  $v=v(t,x)$ be a given smooth function satisfying 
\begin{itemize}
\item[(H1)]  $v$ belongs to $C^0([0,\infty)\times\bar{\Omega};\R^3)\cap C^1([0,\infty)\times\Omega;\R^3)$; $v$ is  Lipschitz continuous in $x$ on $[0,\infty)\times\bar{\Omega}$; $v$ is twice partially differentiable in $x$; all of the partial derivatives of $v$ belong to $C^0([0,\infty)\times \Omega;\R^3)$,
\item[(H2)]  $\pm v$ is subtangential to $\p\Omega$, i.e., $\pm v(s,x) \in T_{\p \Omega}(x)$ for all $(s,x) \in [0,\infty)\times \p\Omega$ with  
\begin{align*}
T_{\p \Omega}(x):=\left\{z \in \R^3\,\Big|\, \liminf\limits_{h \to 0+} \frac{ {\rm dist}\, (x+ hz,\p \Omega)}{h}=0\right\} \text{ for } x \in \p \Omega.
\end{align*}
\end{itemize}
Here are remarks on the above hypothesis: we do not assume any regularity of $\p\Omega$;   
the necessary  regularity of $v$ in the case of \eqref{mod-PDE1} is stronger (see \cite{BFS}); 
(H1) implies that $v$ can be smoothly extended to be an element of $C^1(\R\times\Omega;\R^3)$ as 
$$v(t,x):=-v(-t,x)+2 v(0,x)\quad (\forall\, t\le0,\,\,\,\forall\,x\in\bar{\Omega}),$$
and \eqref{ODE} makes sense for $s<0$; such an  extension is always assumed throughout this paper;  the upcoming method of characteristics involved with $v$ works also for $t<0$, where  Appendix 1 in \cite{BFS} will be exploited (Appendix 1 is also based on such extension of $H$ for $t<0$); (H2) implies that $\p\Omega$ and hence $\Omega$ are forward and backward  flow invariant with respect to the flow map $X$ of \eqref{ODE} (see Lemma 2.1 and Appendix 2 in \cite{BFS} for some explanations on ODEs defined on a closed set and flow invariance); (H2) includes the non-penetration condition and non-slip condition in fluid dynamics; (H2) is not necessarily fulfilled for $t<0$ by the extended $v$.      

Due to (H1) and (H2), the material interface $\interface(t)$ stated in the introduction is well-defined for all $t\ge0$ and up to some $t<0$, within which $\interface(t)$ never touches $\p\Omega$.  
The method of characteristics applied to \eqref{LT} shows that the solution \eqref{LT-sol} is $C^2$-smooth for $C^2$-smooth initial data $\phi^0$. Hence  $\interface(t)$ is a closed $C^2$-smooth surface for each $t$. 

Let $\interface_\ep(t)$ with $\ep>0$ denote the $\ep$-neighborhood of $\interface(t)$, i.e.,
$$\interface_\ep(t):=\bigcup_{x\in\interface(t)}\{y\in\R^3\,|\,|x-y|<\ep\}.$$
{\it A set $\Theta\subset [0,\infty)\times\Omega$ is said to be a ($t$-global) tubular neighborhood of  $\{\interface(t)\}_{t\ge0}$, if  $\Theta$ contains
$$\bigcup_{t\ge0}\Big(\{t\}\times \interface(t)\Big),$$
$\{x\in\Omega\,|\,(t,x)\in\Theta\}$ is an open set for each $t\ge0$ and there exists a nonincreasing function $\ep:[0,\infty)\to\R_{>0}$ such that
$$\{t\}\times\interface_{\ep(t)}(t)\subset\Theta\quad (\forall\,t\ge0).$$}
\indent We introduce the metric projection $\mathcal{P}^t$ from $\Omega$ to $\interface(t)$ as 
$$\mathcal{P}^tx:=\Big\{ \xi\in\interface(t)\,\Big|\,\, |x-\xi|= \inf_{y\in\interface(t)}|x-y|\Big\}\,\,\,\,\mbox{\quad for each $x\in\Omega$, $t\ge0$},$$
where $\mathcal{P}^t$ is well-defined entirely on $\Omega$ as a set-valued mapping. 
Note that  $\mathcal{P}^tx\neq\emptyset$ for all $x\in \Omega$,  since $\interface(t)$ is compact. 
If $\mathcal{P}^tx$ is a singleton, i.e.,  $\mathcal{P}^tx=\{\xi\}$, we simply denote it $\mathcal{P}^tx=\xi$. 
We state basic properties of $\mathcal{P}^t$, including that  $\mathcal{P}^t$ acts  locally near $\interface(t) $ as  the orthogonal projection onto $\interface(t)$:  
\begin{Lemma}\label{Prop001}
For each $t\ge0$,  $x\in\Omega\setminus\interface(t)$ and $\xi\in \mathcal{P}^t x$, the line segment connecting $x$ and $\xi$ is orthogonal to $\interface(t)$ at $\xi$. Furthermore, there exists $\ep=\ep(t)>0$ such that $\mathcal{P}^tx$ is a singleton for every $x$ belonging to $\interface_\ep(t)$.
%
\end{Lemma}
\begin{proof}
Fix an arbitrary $t\ge0$. Take $x\in\Omega\setminus\interface(t)$ and $\xi\in \mathcal{P}^t x$. 
Since  $\interface(t)$ is a closed $C^2$-surface, there exists the tangent plane $T_\xi\interface(t)$  at $\xi$. 
Let $q:(-\delta,\delta)\to \interface(t)$, $\delta>0$ be any smooth curve such that $q(0)=\xi$. 
Note that $q'(0)\in T_\xi\interface(t)$. 
By definition, it holds that 
$$|x-\xi|^2\le |x-q(s)|^2=|x-\xi-(q(s)-q(0))|^2\quad (\forall\,s\in(-\delta,\delta)),$$
which leads to 
$$ 2(x-\xi)\cdot(q(s)-q(0)) \le|(q(s)-q(0))|^2\quad (\forall\, s\in(-\delta,\delta)). $$
Dividing the inequality by $s\neq0$ and sending $s\to 0\pm$, we obtain 
$$(x-\xi)\cdot q'(0)\le 0,\quad (x-\xi)\cdot q'(0)\ge 0.$$ 
This shows the first assertion. 

Since the solution $f(t,\cdot)$ of \eqref{LT} is $C^2$-smooth, $\nu(t,x):=-\nabla f(t,x)/|\nabla f(t,x)|$ is $C^1$-smooth and $\nu(t,x)|_{x\in\interface (t)}$ is  the unit outer normal field of $\interface(t)$. 
Since $\interface(t)$ is compact, we find $\ep_1>0$ and $M>0$ such that $\interface_{\ep_1}(t)\subset\Omega$ and $|\frac{\p \nu_i}{\p x_j}(t,x)|\le M$ ($i,j=1,2,3$) for all $x\in\interface_{\ep_1}(t)$. 
Let $\ep\in(0,\ep_1/2]$ be a constant such that
\begin{align}\label{ppp11}
 1-\{3\cdot(2 M\ep)+6\cdot(2M\ep)^2+6\cdot(2M\ep)^3 \}>0 .  
 \end{align}
Suppose that there exists a point $x\in \interface_\ep(t)$ such that $\mathcal{P}^tx$ is not a singleton. 
Then, we have two different points $\xi,\zeta\in\mathcal{P}^t x\subset \interface(t)$ and $s\in(-\ep,\ep)$ such that  
\begin{align*}
 x-\xi=s\nu(t,\xi),\,\,\, x-\zeta=s\nu(t,\zeta),\,\,\,
|\xi-\zeta|\le|\xi-x|+|x-\zeta|<2\ep\le \ep_1.
\end{align*} 
 Hence, we see that the line segment connecting $\xi$ and $\zeta$ is included in $\interface_{\ep_1}(t)$ and that 
 \begin{align*}
&\nu(t,\xi)-\nu(t,\zeta)=\int_0^s\frac{{\rm d}}{{\rm d}s}\nu\big(t,s(\xi-\zeta)\big){\rm d}s
=\int_0^s\nabla\nu\big(t,s(\xi-\zeta)\big){\rm d}s(\xi-\zeta)=:A(\xi-\zeta),\\ 
&0=\xi-\zeta+s\{\nu(t,\xi)-\nu(t,\zeta)\}=(I+sA)(\xi-\zeta), 
\end{align*}
where the $3\times3$-matrix $A=[A_{ij}]$ is such that $|A_{ij}|\le M$ for all $i,j=1,2,3$. The matrix $I+sA$ is invertible due to \eqref{ppp11}, and therefore we reach a contradiction. 
\end{proof}
As stated in the introduction, the original idea of the  velocity extension method is to  introduce a new velocity field $\tilde{v}$ in a tubular neighborhood $\Theta^0$ of $\{\interface(t)\}_{t\ge0}$ that is an  extension of $v(t,\cdot)|_{\interface(t)}$ having constant values along the normal directions of $\interface(t)$, i.e., 
\begin{align}\label{2DD}
\tilde{v}:\Theta^0\to\R^3,\,\,\,\tilde{v}(t,x):=v(t,\mathcal{P}^tx), 
\end{align}
and then solve the linear transport equation in a tubular neighborhood $\Theta$ of $\{\interface(t)\}_{t\ge0}$ contained in $\Theta^0$:    
\begin{align}\label{Sethian-LT}
&&\left\{
\begin{array}{lll}\medskip
&\dis \frac{\p \phi}{\p t}(t,x)+\tilde{v}(t,x)\cdot\nabla\phi(t,x)=0\quad\mbox{in $\Theta|_{t>0}$}\\ 
&\phi=\phi^0\quad\mbox{on $\Theta|_{t=0}$} 
\end{array}
\right.
\end{align} 
with initial data $\phi^0$ being the local signed distance function of $\interface(0)$. 

While the ``linear'' problem \eqref{Sethian-LT} is standard in numerics, rigorous analysis to obtain Claim 1 directly from \eqref{Sethian-LT}  seems to be difficult.    
In the first place, in order to obtain a well-defined (i.e., single-valued) velocity field $\tilde{v}$ as \eqref{2DD},  we have to assume that  there exists  a tubular neighborhood $\Theta^0$ of $\{\interface(t)\}_{t\ge0}$ in which $\mathcal{P}^tx$ is a singleton for all $(t,x)\in\Theta^0$.  
This is a property of the individual surfaces $\interface(t)$, and it is shown to hold true locally near each $\interface(t)$, if the normal field is itself $C^1$-smooth, c.f., Lemma \eqref{Prop001} above and also Chapter 2.3 in \cite{pruss}. 
Next, in order to obtain joint continuity of $\mathcal{P}^tx$ in $(t,x)$, the family  $\{\interface(t)\}_{t\ge0}$ needs to have regularity in $t$. For instance, if $t\mapsto\interface(t)$ and the field of the normal bundles are continuous with respect to the Hausdorff metric, and all of $\interface(t)$ are $C^2$-surfaces, then $\mathcal{P}^tx$ is continuous in $(t,x)$ locally near $\cup_{t\ge0}\big(\{t\}\times\interface(t)\big)$; see Chapter 2.5 in   \cite{pruss}. 
Here, such regularity is not a priori clear from (H1) and (H2), and in order to solve \eqref{Sethian-LT} in the classical sense, we have to obtain even  stronger regularity of $\mathcal{P}^tx$, say $C^1$-smoothness, as part of the solution process.   
If we assume regularity of  $\mathcal{P}^tx$, the following statement would be expected: 
\medskip
  
\noindent{\bf Claim 1'.} {\it Suppose that  there exists  a tubular neighborhood $\Theta^0$ of $\{\interface(t)\}_{t\ge0}$  such that $(t,x)\mapsto\mathcal{P}^tx$ is single-valued and $C^1$-smooth in $\Theta^0$.  Then there exists a tubular neighborhood $\Theta\subset\Theta_0$ of $\{\interface(t)\}_{t\ge0}$ and a unique $C^2$-solution $\phi$ of \eqref{Sethian-LT}, which satisfies 
\begin{align}\label{1pre}
|\nabla \phi|\equiv1\mbox{ in $\Theta$.}
\end{align}}
\indent We examine Claim 1'.  In this statement, $\tilde{v}$ is assumed to be $C^1$-smooth and the flow map $\tilde{X}(s,\tau,\xi)$ of the ODE $x'(s)=\tilde{v}(s,x(s))$ is locally well-defined.  Since $\tilde{v}=v$ on $\{t\}\times\interface(t)$ for all $t\ge0$, 
 $x(s):=X(s,0,\xi)$ with $\xi \in\interface(0)$ satisfies $x(s)\in \interface(s)$ for all $s\ge 0$ and $x'(s)=\tilde{v}(s,x(s))$. 
 Hence, $\tilde{X}(t,0,\xi)=X(t,0,\xi)$  for all $\xi\in\interface(0)$ and $t\ge0$.  
Furthermore, $\Theta\subset\Theta^0$ exists in such a way that $\Theta$ consists of characteristic curves starting from $s=0$ defined by $\tilde{X}$. 
This is indeed possible:      
due to the continuity of $\tilde{X}$, we find $\ep_1>0$ such that 
$$\bigcup_{0\le s\le 1} \Big(\{s\}\times \tilde{X}(s,0,\interface_{\ep_1}(0))\Big)\subset \Theta^0;$$
then, we find $\ep_2>0$ such that 
$$\interface_{\ep_2}(1)\subset  \tilde{X}(1,0,\interface_{\ep_1}(0)),\quad   
\bigcup_{1\le s\le 2} \Big(\{s\}\times \tilde{X}(s,1,\interface_{\ep_2}(0))\Big)\subset \Theta^0;$$
repeating this process, we find a tubular neighborhood $\Theta$ of $\{\interface(t)\}_{t\ge0}$ that is contained in $\Theta^0$ such that 
$$\tilde{X}(s,t,x)\in \Theta\quad \mbox{ ($\forall\,(t,x)\in\Theta$, $\forall\,s\in[0,t]$)}.$$ 
Within $\Theta$, the method of characteristics is applicable to \eqref{Sethian-LT}, yielding the solution formula $\phi(t,x)=\phi^0(\tilde{X}(0,t,x))$, where we note that the $C^2$-regularity is not directly visible from the right-hand side.
We immediately see that 
\begin{enumerate}
\item $\interface^{\phi}(t):=\{x\in\Omega\,|\,\phi(t,x)=0\}=\interface(t)$, $\phi(t,\cdot)>0$ in $\Theta\cap\Omega^+(t)$ and $\phi(t,\cdot)<0$ in $\Theta\cap\Omega^-(t)$ for each $t\ge0$.  
\item The outer (from $\Omega^+(t)$ to $\Omega^-(t)$) unit normal field $\nu(t,\cdot)$ of $\interface(t)$ coincides with  $-\frac{\nabla\phi(t,\cdot)}{|\nabla\phi(t,\cdot)|}$ on $\interface(t)$.
\item For each $x\in\interface(t)$, the directional derivative of $\tilde{v}(t,\cdot)$ at $x$ along $\pm\nu(t,x)$ is equal to zero, i.e., $\nabla\tilde{v}(t,x)\nu(t,x)=0$, since $\tilde{v}(t,x+a\nu(t,x))\equiv v(t,x)$ for all $|a|\ll 1$.    
\item Along each characteristic curve $x(s):=\tilde{X}(s,t,x)$ with $x\in\interface(t)$, $s\in[0,t]$, it holds that $x(s)\in\interface(s)$ for all $s\in[0,t]$ and 
\begin{align}\label{11pre}
&\frac{\rm d}{{\rm d}s}|\nabla \phi(s,x(s))|^2=-2\Big\langle\nabla\tilde{v}(s,x(s))\nabla \phi(s,x(s)),\nabla \phi(s,x(s))\Big\rangle=0,\\\label{12pre}
& |\nabla \phi(0,x(0))|=|\nabla\phi^0(x(0))|=1,
\end{align}
which leads to  $|\nabla \phi(t,x)|\equiv1$ on $\interface(t)$ for all $t\ge0$.
\end{enumerate}
In order to verify \eqref{1pre}, one must further confirm \eqref{11pre} for all characteristic curves starting from  a small neighborhood of $\interface(0)$, where we note that \eqref{12pre} is fulfilled, since $\phi^0$ is taken to be the local signed distance function of $\interface(0)$.  
The question is then whether or not  $\nabla\tilde{v}(s,x(s))\nabla \phi(s,x(s))=0$ along each characteristic curve $x(s)$ away from $\interface(s)$.  Since $\nabla\tilde{v}(s,x(s))\nabla \phi(s,x(s))$ is the directional derivative of $\tilde{v}(s,\cdot)$ at $x(s)$ in the direction of $\nabla \phi(s,x(s))$, the definition of $\tilde{v}$ implies that $\nabla\tilde{v}(s,x(s))\nabla \phi(s,x(s))=0$ if  
\begin{align}\label{13con}
\nabla \phi(t,x)=\nabla \phi(t, \mathcal{P}^tx)\quad(\forall\,(t,x)\in\Theta).
\end{align} 
As we will see later in Proposition \ref{Prop11}, the identity \eqref{13con} is a consequence of \eqref{1pre},  which looks like {\it circulus in probando}.    
 On top of verification of the regularity hypothesis on $\mathcal{P}^tx$, this is a major nontrivial aspect of  the velocity extension method, particularly in view of  the ``linear'' problem \eqref{Sethian-LT}. 
 Thus,  we will take  a different approach to prove  Claim 1. 

Our strategy is the following:  assuming Claim 1' to be true, we may derive a fully nonlinear first-order PDE (Hamilton-Jacobi equation) without containing the projection $\mathcal{P}^t$,  which is equivalent to \eqref{Sethian-LT}; then we prove wellposedness of the PDE in the class of classical solutions (locally in space and global in time) based on the idea in \cite{BFS} and independently from Claim 1'. 
We will see that such a solution coincides with the function in Claim 1, thus rigorously justifying the  velocity extension method.     
For this purpose, we first observe the following fact: 
\begin{Prop}\label{Prop11}
Suppose that there exists a tubular neighborhood $\Theta$ of $\{\interface(t)\}_{t\ge0}$ and a $C^2$-function $\phi:\Theta\to\R$ such that 
\begin{align}\label{1G}
\left\{
\begin{array}{lll}
  |\nabla\phi|&\equiv& 1\mbox{\quad in $\Theta$},\\
  \phi(t,\cdot)&>&0\mbox{\quad in $\Theta\cap\Omega^+(t)$ for each $t\ge0$}, \\
 \phi(t,\cdot)&<&0\mbox{\quad in $\Theta\cap\Omega^-(t)$  for each $t\ge0$}.
 \end{array}
 \right.  
\end{align}
Then the following statements hold true for each $t\ge0$: 
\begin{itemize}
\item[\rm(i)] For each $\xi\in\interface(t)$, $\nabla\phi(t,\cdot)$ is constant on the line segment $$l(\xi):=\{\xi+s\nabla\phi(t,\xi)\,|\,s\in\R\}\cap \{x\in\Omega\,|\,(t,x)\in\Theta  \}.$$ 
\item[\rm(ii)]  $\phi(t,\cdot)$ is  the signed distance function of $\interface(t)$. 
\item[\rm(iii)] In $\Theta$, the metric projection $\mathcal{P}^t$ onto $\interface(t)$ is single-valued and represented as  
\begin{align*}
\mathcal{P}^tx=x-\phi(t,x)\nabla\phi(t,x).
\end{align*}
In particular, $(t,x)\mapsto\mathcal{P}^tx$ is $C^1$-smooth in $\Theta$. 
\end{itemize}
\end{Prop}
\begin{proof}
Note that $\interface^{\phi}(t):=\{x\in\Omega\,|\,\phi(t,x)=0\}=\interface(t)$ for all $t\ge0$.  
Let $s\mapsto\gamma(s)$ be the unique solution of  
$$\gamma'(s)=\nabla\phi(t,\gamma(s)),\quad\gamma(0)=\xi\in\interface(t)\qquad (\mbox{$t$ is now fixed}).$$
As long as  $(t,\gamma(s))\in\Theta$, we have $|\nabla\phi(t,\cdot)|\equiv1$ in a neighborhood of $\gamma(s)$. 
Hence, we obtain  
\begin{align*}
\frac{\rm d}{{\rm d}s}\frac{\p \phi}{\p{x_i}}(t,\gamma(s))&= \nabla\frac{\p \phi}{\p{x_i}}(t,\gamma(s))\cdot \nabla\phi(t,\gamma(s)) 
=\frac{\p}{\p x_i} \nabla\phi(t,\gamma(s))\cdot \nabla\phi(t,\gamma(s))\\
&=\frac{1}{2}\frac{\p}{\p x_i} |\nabla\phi(t,x)|^2\Big|_{x=\gamma(s)} 
=0,
\end{align*}
which yields $\nabla\phi(t,\gamma(s))\equiv\nabla\phi(t,\gamma(0))=\nabla\phi(t,\xi)$. 
We also find that $\gamma(s)=\xi+s\nabla\phi(t,\xi)$, which concludes (i). 

Note that $\phi(t,\cdot)=0$ on $\interface(t)$ and $-\nabla\phi(t,\xi)$ is the outer unit normal of $\interface(t)$ at $\xi\in\interface(t)$. 
For each $(t,x)\in\Theta$, $\mathcal{P}^t x$ is a singleton.  
In fact, if $\xi,\zeta\in \mathcal{P}^t x$, Lemma \ref{Prop001} implies that there exists $s\in\R$ such that $x-\xi=s\nabla\phi(t,\xi)$, $x-\zeta=s\nabla\phi(t,\zeta)$ and (i) implies that $\nabla\phi(t,\xi)=\nabla\phi(t,x)=\nabla\phi(t,\zeta)$, which yields $\xi=\zeta$.     
Hence, with $\xi:= \mathcal{P}^t x\in\interface(t)$,  it holds that  $x-\xi=s\nabla\phi(t,\xi)$ with some $s\in\R$ and 
$$\phi(t,x)=\phi(t,x)-\phi(t,\xi)=\nabla \phi(t,\tilde{x})\cdot(x-\xi)=\nabla \phi(t,\tilde{x})\cdot s\nabla\phi(t,\xi),$$ 
where  $\tilde{x}$ is some point on $l(\xi)$. 
By (i), we obtain $\nabla\phi(t,\tilde{x})\cdot\nabla\phi(t,\xi)=\nabla\phi(t,\xi)\cdot \nabla\phi(t,\xi)=1$ and 
$$\phi(t,x)=s={\rm sign}(s)|x-\xi|={\rm sign}(s)\min_{y\in\interface(t)}|x-y|,$$
where ${\rm sign}(s)=\pm 1$ if $x\in \Omega^\pm(t)$. Now (ii) is proven. 

For each $(t,x)\in\Theta$ and  $\xi:= \mathcal{P}^t x$, the above reasoning implies that  
\begin{align*}
\mathcal{P}^t x=\xi=x- s\nabla\phi(t,\xi)=x-\phi(t,x) \nabla\phi(t,x).
\end{align*}
\end{proof}
Proposition \ref{Prop11} shows that the function $\phi$ in Claim 1' and the one in Claim 1 are indeed the local distance function of $\interface (t)$ for each $t\ge0$.  
Furthermore,  by Proposition \ref{Prop11}, we obtain:     
\begin{Prop}\label{Proppp1122}
Suppose that Claim 1'  holds true. Then the solution $\phi$ to \eqref{Sethian-LT} solves 
\begin{align}\label{S-HJ}
&&\left\{
\begin{array}{lll}\medskip
&\dis \frac{\p \phi}{\p t}(t,x)+v\Big(t,x-\phi(t,x)\nabla\phi(t,x)\Big)\cdot\nabla\phi(t,x)=0\quad\mbox{in $\Theta|_{t>0}$}\\&\phi=\phi^0\quad\mbox{on $\Theta|_{t=0}$}.
\end{array}
\right.
\end{align}
\end{Prop}  
We have thus observed that Claim 1' leads to the solvability of \eqref{S-HJ}.       
We will rigorously derive  not only Claim 1' but also  Claim 1 only with (H1) and (H2), by  directly finding $\Theta$ and $C^2$-function satisfying \eqref{S-HJ} and \eqref{1G}.   
Here are the main results of this paper. 
\begin{Thm}\label{main-thm1}
Suppose that initial data $\phi^0$ satisfying   \eqref{2initial}  is $C^2$-smooth with $|\nabla \phi^0|\equiv1$ in a neighborhood of $\interface(0)$ (hence, $\phi^0$ locally coincides with the local signed distance function of $\interface(0)$). 
Then there exist a $t$-global tubular neighborhood $\Theta$ of $\{\interface(t)\}_{t\ge0}$ and  a $C^2$-function $\phi:\Theta\to\R$ solving  \eqref{S-HJ}. Furthermore, $\phi$ satisfies \eqref{1G}. 
\end{Thm}  
\noindent The next corollary immediately follows from  Theorem \ref{main-thm1} and Proposition \ref{Prop11}: 
\begin{Cor}[justification of Claim 1]\label{main-thm1-1}
Let $\Theta$ and $\phi$ be as stated in Theorem \ref{main-thm1}. Then $(t,x)\mapsto\mathcal{P}^tx$ is single-valued and $C^1$-smooth.  Furthermore, $\phi$   solves \eqref{Sethian-LT} and $\phi(t,\cdot)$ is the local signed distance function of $\interface(t)$ for each $t\ge0$.
\end{Cor}  
We generalize Theorem \ref{main-thm1} so that one can use initial data which does not necessarily coincide with the local signed distance function of $\interface(0)$. 
\begin{Thm}\label{main-thm2}
Suppose that initial data $\phi^0$ satisfying \eqref{2initial}  is $C^2$-smooth in a neighborhood of $\interface(0)$. 
Then there exist a $t$-global tubular neighborhood $\Theta$ of $\{\interface(t)\}_{t\ge0}$ and  a $C^2$-function $\phi:\Theta\to\R$ solving  
\begin{align}\label{S-HJ2}
&&\left\{
\begin{array}{lll}\medskip 
&\dis \frac{\p  \phi}{\p t}(t,x)+v\Big(t,x-\frac{\phi(t,x)\nabla\phi(t,x)}{|\nabla\phi(t,x)|^2}\Big)\cdot\nabla\phi(t,x)=0\quad\mbox{in $\Theta|_{t>0}$}\\&\phi=\phi^0\quad\mbox{on $\Theta|_{t=0}$}.
\end{array}
\right.
\end{align}
Furthermore, it holds that 
\begin{itemize}
\item $\phi(t,\cdot)>0\mbox{ in $\Theta\cap\Omega^+(t)$ and $\phi(t,\cdot)<0$ in $\Theta\cap\Omega^-(t)$  for each $t\ge0$}$; 
\item $\forall\, (t,x)\in \Theta$ $\exists\,\xi$ such that $(0,\xi)\in\Theta$ and  $|\nabla \phi(t,x)|=|\nabla\phi^0(\xi)| \neq 0$. 
 \end{itemize} 
\end{Thm}  
\noindent We recall that \eqref{mod-PDE1} studied in \cite{BFS} preserves  $|\nabla\phi|$ only on the interface as \eqref{1111pre}. 

It is clear that  Theorem  \ref{main-thm1} follows from Theorem  \ref{main-thm2}.   
Theorem \ref{main-thm2}  deals with phenomena in a neighborhood of the interface. 
Now we consider the issue within the whole domain $\Omega$.
For this purpose,  in the rest of this section, we deal with an extension of $v$ over $[0,\infty)\times\R^3$, still denoted by the same symbol,  based on a classical result on extension of Lipschitz functions (our case is slightly different due to the presence of $t$). 
The next lemma guarantees that our velocity field $v$ satisfying (H1) can be extended over $[0,\infty)\times\R^3$ keeping continuity in $(t,x)$ and Lipschitz continuity in $x$.  
\begin{Lemma}\label{2exten}
Suppose that $v=(v_1,v_2,v_3):[0,\infty)\times\bar{\Omega}\to\R^3$ is continuous in $(t,x)$ and $v_1,v_2,v_3$ are Lipschitz continuous in $x$ with a Lipschitz constant $\lambda\ge0$. 
Then, there exists $\tilde{v}=(\tilde{v}_1,\tilde{v}_2,\tilde{v}_3):[0,\infty)\times\R^3\to\R^3$ being continuous in $(t,x)$ such that  $\tilde{v}\equiv v$ on  $[0,\infty)\times\bar{\Omega}$ and $\tilde{v}_1,\tilde{v}_2,\tilde{v}_3$ are Lipschitz continuous in $x$ with the same Lipschitz constant $\lambda$. The extension is explicitly given for each $(t,x)\in[0,\infty)\times\R^3$ as 
$$\tilde{v}_i(t,x):=\inf_{z\in\bar{\Omega}}\{  v_i(t,z)+\lambda|x-z| \}\quad (i=1,2,3).$$ 
\end{Lemma}
\begin{proof}
See Appendix. 
\end{proof}
\noindent We remark that, if $v|_{\p\Omega}=0$  (this is the non-slip boundary condition in fluid mechanics satisfying (H2)),  our extension is simply given as the $0$-extension outside $\bar{\Omega}$,  where $v$ is still continuous in $(t,x)$ and Lipschitz continuous in $x$ on $[0,\infty)\times\R^3$. 

In order to allow $|\nabla \phi|$ to be $0$ somewhere in $\Omega$, we modify the term  $x-\frac{\phi(t,x)\nabla\phi(t,x)}{|\nabla\phi(t,x)|^2}$ as 
$$x-\frac{\phi(t,x)\nabla\phi(t,x)}{|\nabla\phi(t,x)|^2+\eta(|\nabla\phi(t,x)|^2)}$$
with a $C^1$-function $\eta:[0,\infty)\to[0,1]$ such that 
 \begin{align*}
\eta(r):=\left\{
\begin{array}{lll}
&0&\quad\mbox{for $r\ge r_\ast$ with some fixed $r_\ast\in(0,1)$}\\
&1&\quad\mbox{for $r=0$}\\
&\mbox{monotone decreasing}&\quad\mbox{otherwise}.
\end{array}
\right.
\end{align*}
\begin{Thm}\label{Thm-viscosity}
Suppose that initial data $\phi^0\in C^0(\bar{\Omega};\R)$ satisfies 
$$\mbox{ $\phi^0>0$ on $\Omega^+(0)$,  \quad$\phi^0<0$ on $\Omega^-(0)$, \quad$\{x\in\Omega\,|\,\phi^0(x)=0\}=\interface(0)$}$$
 and $\phi^0|_{\p\Omega}=0$. 
Then there exists a unique viscosity solution $\phi\in C^0([0,\infty)\times\Omega;\R)$ of 
\begin{align}\label{S-HJm}
&&\left\{
\begin{array}{lll}\medskip
&\dis \frac{\p \phi}{\p t}(t,x)+v\Big(t,x-\frac{\phi(t,x)\nabla\phi(t,x)}{|\nabla\phi(t,x)|^2+\eta(|\nabla\phi(t,x)|^2)}\Big)\cdot\nabla\phi(t,x)=0&\quad\mbox{in $(0,\infty)\times\Omega$}\\ \medskip 
&\phi(0,x)=\phi^0(x)&\quad\mbox{on $\Omega$},\\
&\phi(t,x)=0&\quad\mbox{on $[0,\infty)\times\p\Omega$,}
\end{array}
\right.
\end{align}
i.e., $\phi$ satisfies the PDE in \eqref{S-HJm} in the sense of viscosity solutions and the initial/boundary condition in the classical sense. Furthermore, it holds that 
\begin{align}\label{visin}
 \phi(t,\cdot)>0\mbox{ in $\Omega^+(t)$, $\phi(t,\cdot)<0$ in $\Omega^-(t)$,  $\phi(t,\cdot)=0$ on $\interface(t)$  \quad  ($\forall\,t\ge0$)}.
\end{align}
\end{Thm}
\noindent We remark that $|\nabla\phi^0|_{\interface(0)}\neq0$ is not assumed in Theorem \ref{Thm-viscosity};  
\eqref{S-HJm} has a unique viscosity solution independently from the two-phase flow setting with an interface. 
If we add a cut-off function to the PDE in \eqref{S-HJm} so that the nonlinear term $-\frac{\phi\nabla\phi}{|\nabla\phi|^2+\eta(|\nabla\phi|^2)}$ vanishes in a small neighborhood of $\p\Omega$, i.e., 
$$\frac{\p \phi}{\p t}(t,x)+v\Big(t,x-\tilde{\eta}(t,x)\frac{\phi(t,x)\nabla\phi(t,x)}{|\nabla\phi(t,x)|^2+\eta(|\nabla\phi(t,x)|^2)}\Big)\cdot\nabla\phi(t,x)=0$$
with such a cut-off function $\tilde{\eta}$, 
we may drop the assumption $\phi^0|_{\p\Omega}=0$, where the boundary condition ``$\phi(t,x)=0$ on $[0,\infty)\times\p\Omega$'' must be replaced by  ``$\phi(t,x)=\phi^0(X(0,t,x))$ on $[0,\infty)\times\p\Omega$''  (such a setting is used in \cite{BFS}). 

Finally, we state partial regularity of the viscosity solution obtained in Theorem \ref{Thm-viscosity}.   Suppose that initial data $\phi^0$ satisfies the condition in Theorem \ref{main-thm2}. 
Then there exists $\ep_0>0$ such that 
$$\interface_{\ep_0}(0)\subset\Omega,\,\,\,\inf_{\xi\in\interface_{\ep_0}(0)}|\nabla \phi^0(\xi)|>0,$$
where the solution  $\phi:\Theta\to\R$ obtained in Theorem \ref{main-thm2} is constructed so that $\Theta|_{t=0}\subset \{0\}\times \interface_{\ep_0}(0)$.
Let the constant $r_\ast$ in the definition of $\eta$ be such that  
\begin{align}\label{eta2}
r_\ast\in(0,1),\quad \sqrt{r_\ast}\le \frac{1}{2} \inf_{\xi\in\interface_{\ep_0}(0)}|\nabla \phi^0(\xi)|. 
\end{align}
Then the solution $\phi:\Theta\to0$ stated in Theorem \ref{main-thm2}  locally satisfies the PDE in \eqref{S-HJm} in the classical sense, 
 since $\eta(|\nabla\phi(t,x)|^2)=0$ in $\Theta$.  
Hence, it makes sense to locally compare the solutions in Theorem \ref{main-thm2} and Theorem \ref{Thm-viscosity}. 
\begin{Thm}\label{Thm-comparison}
Suppose that the initial data $\phi^0\in C^0(\Omega;\R)$ satisfies $\phi^0|_{\p\Omega}=0$ and the condition stated in Theorem \ref{main-thm2}. Suppose also \eqref{eta2}. Let $\phi, \Theta$ be as stated in  Theorem \ref{main-thm2} and $\tilde{\phi}$ denote the viscosity solution provided by Theorem \ref{Thm-viscosity}. Then there exists a $t$-global tubular neighborhood $\tilde{\Theta}\subset \Theta$ in which $\phi\equiv\tilde{\phi}$ holds.   
\end{Thm}

\setcounter{section}{2}
\setcounter{equation}{0}
\section{Proof of main results}
We prove Theorem  \ref{main-thm2}--\ref{Thm-comparison} based on the methods developed in \cite{BFS}.
For the readers' convenience, we give self-contained presentations.  
\subsection{Proof of Theorem  \ref{main-thm2}}

The PDE in \eqref{S-HJ2} is written in the form of the Hamilton-Jacobi equation 
\begin{eqnarray}\label{HJ22}
\frac{\p \phi}{\p t}(t,x)+H(t,x,\nabla \phi(t,x),\phi(t,x))=0
\end{eqnarray}
 generated by the Hamiltonian $H$ defined on a subset of $\R\times\R^3\times\R^3\times \R$ as
\begin{align*}
H(t,x,p,\Phi):=v\Big(t,x-\Phi\frac{p}{|p|^2}\Big)\cdot p.
\end{align*}
We recall that, due to (H1), $v$ can be smoothly extended to $t<0$.
Hence, the characteristic system of ODEs to  \eqref{S-HJ2} is well-defined also for $t<0$ (not necessarily within $(-\infty,0]$), which is (implicitly) necessary for  the upcoming application of the inverse function theorem in the  method of characteristics. 
Set 
\begin{align*}
B(p)=B(p_1,p_2,p_3):=\begin{bmatrix}
-p_1^2+p_2^2+p_3^2& -2 p_1p_2 & -2p_1p_3\\
-2p_1p_2& p_1^2-p_2^2+p_3^2&-2 p_2p_3\\
-2p_1p_3&-2p_2p_3&p_1^2+p_2^2-p_3^2
\end{bmatrix}.
\end{align*}   
The characteristic system of ODEs to \eqref{S-HJ2} is given as
\begin{align}\label{2chara}
 x'(s)&= \frac{\p H}{\p p}(s,x(s),p(s),\Phi(s))\\\nonumber
 &=v\Big(s,x(s)-\Phi(s)\frac{p(s)}{|p(s)|^2}\Big) -\frac{\Phi(s)}{|p(s)|^4}\Big[ \nabla v\Big(s,x(s)-\Phi(s)\frac{p(s)}{|p(s)|^2}\Big)B(p(s)) \Big]^\tra p(s), \\\label{2chara2}
 p'(s)&=-\frac{\p H}{\p x} (s,x(s),p(s),\Phi(s))-\frac{\p H}{\p\Phi}(s,x(s),p(s),\Phi(s))p(s)\\\nonumber
&=-\Big[\nabla v\Big(s,x(s)-\Phi(s)\frac{p(s)}{|p(s)|^2}\Big)\Big]^\tra p(s)\\\nonumber
&\quad +\Big<\Big[\nabla v\Big(s,x(s)-\Phi(s)\frac{p(s)}{|p(s)|^2}\Big)\Big]^\tra\frac{p(s)}{|p(s)|^2}, p(s)\Big>p(s),\\\label{2chara3}
\ \Phi'(s)&= \frac{\p H}{\p p}(s,x(s),p(s),\Phi(s))\cdot p(s)-H(s,x(s),p(s),\Phi(s)) \\\nonumber
&=  -\frac{\Phi(s)}{|p(s)|^4}\Big<\Big[ \nabla v\Big(s,x(s)-\Phi(s)\frac{p(s)}{|p(s)|^2}\Big)B(p(s)) \Big]^\tra p(s) ,p(s)\Big>, \\\label{2chara4}
&\!\!\!\!x(0)=\xi,\quad p(0)=\nabla\phi^0(\xi),\quad \Phi(0)=\phi^0(\xi),
\end{align}
where  initial condition \eqref{2chara4} is given only for $\xi$ such that $\nabla\phi^0(\xi)\neq0$, $\xi-\frac{\phi^0(\xi)\nabla\phi^0(\xi)}{|\nabla\phi^0(\xi)|^2}\in\Omega$ (the first one is to avoid $\frac{p(0)}{|p(0)|^2}$ with $p(0)=0$; the second one is necessary since $v$ is not defined outside $\bar{\Omega}$).  The solution of  \eqref{2chara}--\eqref{2chara3} shall be denoted by  $x(s;\xi),p(s;\xi),\Phi(s;\xi) $.
Due to the condition $\nabla\phi^0\neq0$ on $\interface(0)$,  there exists $\ep_0>0$ such that 
$$\interface_{\ep_0}(0)\subset\Omega,\quad \inf_{\xi\in\interface_{\ep_0}(0)}|\nabla \phi^0(\xi)|>0.$$
In \eqref{2chara4}, $\xi$ is taken from $\interface_{\ep_0}(0)$.  
We demonstrate an iteration of the methods of characteristics (the following Step 1--4) within a  shrinking neighborhood of the interface.
%

{\bf Step 1.} {\it  A priori estimates for the characteristic system of ODEs.} 
As long as  the solution of  \eqref{2chara}--\eqref{2chara3} exists (including $s<0$), it holds that
\begin{align}\label{1xsxs}
\Phi(s;\xi)&=\Phi(0;\xi)=0\quad (\forall\, \xi \in\interface(0)),\\\label{2xsxs}
x'(s;\xi)& = v(s,x(s;\xi)), \quad x(s;\xi)\in\interface(s) \quad (\forall\, \xi \in\interface(0)),\\\label{1xsxs2}
{\rm sign}\,\Phi(s;\xi)&={\rm sign}\,\Phi(0;\xi)={\rm sign}\,\phi^0(\xi)\quad({\rm sign}\,0:=1)\quad (\forall\, \xi \in\interface_{\ep_0}(0)),\\\label{2xsxs2}
p'(s;\xi)\cdot p(s;\xi)&=-\Big<\Big[\nabla v\Big(s,x(s)-\Phi(s)\frac{p(s)}{|p(s)|^2}\Big)\Big]^\tra p(s),p(s)\Big>\\\nonumber
&\!\!\!\!\!\!\!\!\!\!\!\!\!\!\!\!\!\!\!\!\!+\Big<\nabla \Big[v\Big(s,x(s)-\Phi(s)\frac{p(s)}{|p(s)|^2}\Big)\Big]^\tra\frac{p(s)}{|p(s)|^2}, p(s)\Big>|p(s)|^2=0 \quad (\forall\, \xi \in\interface_{\ep_0}(0)),\\\label{2xsxs3}
|p(s;\xi)|^2&=|p(0;\xi)|^2=|\nabla \phi^0(\xi)|^2\neq0 \quad (\forall\, \xi \in\interface_{\ep_0}(0)). 
\end{align}
Note that  \eqref{2chara}--\eqref{2chara3} with $\xi\in\interface(0)$ is  solvable for all $s\ge0$ due to (H2) and \eqref{2xsxs}, and up to some $s<0$.
For each $\xi \in\interface(0)$, the variational equations for $\frac{\p x}{\p\xi}(s)=\frac{\p x}{\p\xi}(s;\xi)$ and $\frac{\p \Phi}{\p\xi}(s)=\frac{\p \Phi}{\p\xi}(s;\xi)$ are given as 
\begin{align}\label{1vvv2}
&\frac{\rm d}{{\rm d} s}\frac{\p x}{\p\xi_i}(s)=\nabla v(s,x(s))\frac{\p x}{\p\xi_i}(s)- \nabla v(s,x(s))\frac{p(s)}{|p(s)|^2}\frac{\p\Phi}{\p\xi_i}(s)  \\\nonumber
&\quad  -\Big[ \nabla v(s,x(s))B(p(s)) \Big]^\tra \frac{p(s)}{|p(s)|^4}\frac{\p \Phi}{\p\xi_i}(s),\\ \nonumber
&\frac{\p x}{\p\xi_i}(0)=e^i\quad(i=1,2,3),\\\label{2vvv2}
&\frac{\rm d}{{\rm d} s}\frac{\p\Phi}{\p\xi}(s)=  -\frac{1}{|p(s)|^4}\Big<\Big[ \nabla v(s,x(s))B(p(s)) \Big]^\tra p(s) ,p(s)\Big>\frac{\p \Phi}{\p\xi}(s),\\\nonumber  
&\frac{\p \Phi}{\p\xi}(0)=\nabla\phi^0(\xi),
\end{align}
where $e^1=(1,0,0),e^2=(0,1,0),e^3=(0,0,1)$.  
Hereafter, $U_0,U_1,U_2,U_3,U_4, U_5,U_6,U_7$ will be constants depending only on $v$, $\nabla v$ and  $0<\inf_{\interface(0)}|\nabla\phi^0|\le\sup_{\interface(0)}|\nabla\phi^0|<\infty$. 
Due to (H1) and $|p(s)|=|p(0)|$, as long as  solutions exist,  
it holds that for any $\xi\in\interface(0)$,
\begin{align*}
&|\nabla v(s,x(s)) q|\le U_0|q|,\quad |\langle\nabla v(s,x(s)) q,q\rangle|\le U_1|q|^2\quad (\forall\,q\in\R^3),\\
&\Big|\nabla v(s,x(s))\frac{p(s)}{|p(s)|^2}\Big|\le U_2,\\
&\Big| \Big[ \nabla v(s,x(s))B(p(s)) \Big]^\tra \frac{p(s)}{|p(s)|^4}  \Big|\le U_3,\\
& \Big| \frac{1}{|p(s)|^4}\Big<\Big[ \nabla v(s,x(s))B(p(s)) \Big]^\tra p(s) ,p(s)\Big> \Big|\le U_4.
\end{align*}
Hence, \eqref{2vvv2} provides
\begin{align*}
&\Big|\frac{\p\Phi}{\p\xi}(s)\Big|=\Big|\nabla \phi^0(\xi) \exp\Big(\int_0^s  -\frac{1}{|p(\tau)|^4}\langle[ \nabla v(\tau,x(\tau))B(p(\tau)) ]^\tra p(\tau) ,p(\tau)\rangle  {\rm d}\tau \Big)\Big|\le  U_5e^{U_4|s| }.
\end{align*}
For each $i=1,2,3$ and $\xi \in\interface(0)$,  \eqref{1vvv2} provides
\begin{align*}
\frac{\rm d}{{\rm d} s}\frac{\p x}{\p\xi_i}(s)\cdot \frac{\p x}{\p\xi_i}(s)&=\Big<\nabla v(s,x(s))\frac{\p x}{\p\xi_i}(s), \frac{\p x}{\p\xi_i}(s)\Big> - \Big<\nabla v(s,x(s))\frac{p(s)}{|p(s)|^2}\frac{\p\Phi}{\p\xi_i}(s), \frac{\p x}{\p\xi_i}(s)\Big> \\\nonumber
&\quad  -\Big<\Big[ \nabla v(s,x(s))B(p(s)) \Big]^\tra \frac{p(s)}{|p(s)|^4}\frac{\p \Phi}{\p\xi_i}(s), \frac{\p x}{\p\xi_i}(s)\Big>,\\
\2\frac{\rm d}{{\rm d} s}\Big|\frac{\p x}{\p\xi_i}(s)\Big|^2
&\le U_1\Big|\frac{\p x}{\p\xi_i}(s)\Big|^2+U_2U_5e^{U_4|s|}\Big|\frac{\p x}{\p\xi_i}(s)\Big|+U_3U_5e^{U_4|s|}\Big|\frac{\p x}{\p\xi_i}(s)\Big|\\
&\le  (1+U_1) \Big|\frac{\p x}{\p\xi_i}(s)\Big|^2+  \frac{U_2^2+U_3^2}{2}U_5^2e^{2U_4|s|}.
\end{align*}
Gronwall's inequality yields 
\begin{align}\label{est21}
\Big|\frac{\p x}{\p\xi_i}(s;\xi)\Big|^2&\le U^2_6\quad  (\forall\, s\in[0,1],\,\,\,\forall\,\xi\in\interface(0)).
\end{align}
It follows from \eqref{1vvv2} and \eqref{est21} that
\begin{align}\label{est22}
&\dis \Big|\frac{\rm d}{{\rm d} s}\frac{\p x}{\p\xi_i}(s;\xi)\Big|\le U_0 \Big|\frac{\p x}{\p\xi_i}(s;\xi)\Big|+(U_2+U_3)U_5 e^{U_4|s|}\le U_7 \quad (\forall\, s\in[0,1],\,\,\,\forall\,\xi\in\interface(0)).
\end{align}
The continuity in \eqref{2chara}--\eqref{2chara4} with respect to initial data and the estimates  \eqref{est21}--\eqref{est22} imply that there exist $\ep_1\in(0,\ep_0]$ and $0<\delta\ll1$ such that
\begin{eqnarray}\nonumber
&&\mbox{\eqref{2chara}-\eqref{2chara4} is solvable  for  all $s\in[-\delta,1+\delta]$ and  all $\xi\in \interface_{\ep_1}(0)$},\\\label{est23}
&&
\Big|\frac{\rm d}{{\rm d} s}\frac{\p x}{\p\xi_i}(s;\xi)\Big|< 2U_7 \quad (\forall\, s\in[-\delta,1+\delta],\,\,\, \forall\,\xi\in\interface_{\ep_1}(0)).
\end{eqnarray}
Let $t_\ast\in(0,1]$ be a number such that
\begin{eqnarray}\label{t-star}
1- \{3\cdot (2U_7 t_\ast)+6\cdot (2U_7 t_\ast)^2+6\cdot(2U_7 t_\ast)^3\}>0.
\end{eqnarray}
We make the above $\delta>0$ smaller so that 
$$1- \{3\cdot (2U_7 |s|)+6\cdot (2U_7 |s|)^2+6\cdot(2U_7 |s|)^3\}>0\quad (\forall\,s\in[-\delta,t^\ast+\delta]).$$ 
\indent {\bf Step 2.} {\it The injectivity of $\xi\mapsto x(s;\xi)$, i.e.,}    
\begin{align}\label{2contra}
&\mbox{\it \quad  $\exists\, \tilde{\ep}_1\in(0,\ep_1]$ s.t.  $\interface_{\tilde{\ep}_1}(0)\ni\xi\mapsto x(s;\xi)$ is injective for each $s\in[-\delta, t_\ast+\delta]$}.
\end{align}
We  prove \eqref{2contra} by contradiction.
Suppose that \eqref{2contra} does not hold. Then, there exists  a sequence $\{\nu_j\}_{j\in\N}\subset(0,\ep_1]$  with $\nu_j\to0$ as $j\to\infty$ and $\{t_j\}_{j\in\N}\subset[-\delta,t_\ast+\delta]$ such that for each $j\in\N$  
\begin{eqnarray}\label{22contra}
\exists\,\zeta,\tilde{\zeta}\in \interface_{\nu_j}(0)\mbox{ such that }\zeta\neq\tilde{\zeta},\,\,\,x(t_j;\zeta)=x(t_j;\tilde{\zeta}).
\end{eqnarray}
Set $d_j :=\sup\{ |\tilde{\zeta}-\zeta|\,|\,  \eqref{22contra}\mbox{ holds}  \}$. We check that $d_j\to0$ as $j\to\infty$. If not, there exist $\eta>0$, $s\in[-\delta,t^\ast+\delta]$ and a subsequence of $\{(d_j,t_j)\}_{j\in\N}$ (still denoted by the same symbol) such that $d_j\ge \eta$ for all $j$ and $t_j\to s$ as $j\to\infty$.
Then, for each $j\in\N$, we find  $\zeta_j$ and $\tilde{\zeta}_j$ such that
$$|\zeta_j-\tilde{\zeta}_j|\ge \frac{\eta}{2},\quad {\rm dist}(\zeta_j,\interface(0))\le \nu_j,\quad
{\rm dist}(\tilde{\zeta}_j,\interface(0))\le \nu_j.
$$
Up to a subsequence, $(\zeta_j, \tilde{\zeta}_j)$ converges to some $(\xi,\tilde{\xi})\in \interface(0)\times \interface(0)$ with $|\xi-\tilde{\xi}|\ge \frac{\eta}{2}$ as $j\to\infty$. 
The limit $j\to\infty$ in $x(t_j;\zeta_j)=x(t_j;\tilde{\zeta}_j)$ leads to 
$$\mbox{$x(s;\xi)=x(s;\tilde{\xi})$ with $\xi,\tilde{\xi}\in\interface(0)$, $\xi\neq\tilde{\xi}$},$$
which is a contradiction due to the unique solvability of  \eqref{2xsxs}. 
Thus,  $d_j\to0$ as $j\to\infty$. 
For each $j$ sufficiently large, $\zeta=\zeta_j,\tilde{\zeta}=\tilde{\zeta}_j$ in \eqref{22contra} are included in the $\ep_1$-neighborhood of some $\xi_j\in\interface(0)$. 
 For $i=1,2,3$,  Taylor's approximation within the $\ep_1$-neighborhood of $\xi_j$ provides with some $\lambda_{ji},  \lambda'_{ji}\in (0,1)$, 
\begin{eqnarray*}
0&=&x_i(t_j;\tilde{\zeta}_j)-x_i(t_j;\zeta_j)\\
&=&x_i(0;\tilde{\zeta}_j)-x_i(0;\zeta_j)+(x_i(t_j;\tilde{\zeta}_j)- x_i(t_j;\zeta_j))-(x_i(0;\tilde{\zeta}_j)-x_i(0;\zeta_j))\\
&=&(\tilde{\zeta}_j-\zeta_j)_i+\frac{{\rm d}x_i}{{\rm d} s}(\lambda_{ji}t_j;\tilde{\zeta}_j)t_j- \frac{{\rm d} x_i}{{\rm d} s}(\lambda_{ji}t_j;\zeta_j))t_j, \\
&=&(\tilde{\zeta}_j-\zeta_j)_i+\frac{\rm d}{{\rm d} s}\frac{\p x_i}{\p\xi}(\lambda_{ji}t_j;\tilde{\zeta}_j +\lambda'_{ji}(\tilde{\zeta}_j -\zeta_j ) )t_j\cdot (\tilde{\zeta}_j-\zeta_j),
\end{eqnarray*}
where $(y)_i$ stands for the $i$-th component of $y\in\R^3$. 
We obtain with \eqref{est23}, 
\begin{align*}
(I+A)(\tilde{\zeta}_j-\zeta_j)=0,
\end{align*}
where $I$ is the  ($3\times3$)-identity matrix,  $A$ is a ($3\times3$)-matrix with $|A_{kl}|<2U_7|t_j|$, and hence 
\begin{align*}
\det(I+A)\ge1- \{3\cdot (2U_7|t_j|)+6\cdot (2U_7 |t_j|)^2+6\cdot(2U_7 |t_j|)^3\}>0,
\end{align*}
which implies that $\tilde{\zeta}_j=\zeta_j$. 
This is a contradiction, and therefore \eqref{2contra} holds. 
Note that $\det \frac{\p x}{\p \xi}(s;\xi)>0$ for all $s\in[-\delta,t_\ast+\delta]$ and $\xi\in\interface_{\tilde{\ep}_1}(0)$, since 
$$ \mbox{$\dis\frac{\p x}{\p \xi}(s;\xi)= \frac{\p x}{\p \xi}(0;\xi)+ \frac{\p x}{\p \xi}(s;\xi)- \frac{\p x}{\p \xi}(0;\xi)=I+\Big[\frac{\rm d}{{\rm d} s}  \frac{\p x_k}{\p \xi_l}(\lambda_{kl}s;\xi)s\Big]$}$$
with some $\lambda_{kl}\in(0,1)$. 

\indent {\bf Step 3.} {\it Construction of a local in space solution within $[0,t_\ast]$.} Define
\begin{eqnarray*}
&& O_1 :=\bigcup_{-\delta< s< t_\ast+\delta}\Big( \{s\}\times \{x(s;\xi)\,|\,\xi\in\interface_{\tilde{\ep}_1}(0)\}\Big),\\
&&\psi_1: (-\delta,t_\ast+\delta)\times \interface_{\tilde{\ep}_1}(0)\to  O_1,\quad \psi_1(s,\xi):=(s, x(s;\xi)),\\
&&\psi_1(t,\cdot):  \interface_{\tilde{\ep}_1}(0)\to  O_1|_{s=t}=\{t\}\times \{x(t;\xi)\,|\,\xi\in\interface_{\tilde{\ep}_1}(0)\}\quad (-\delta< \forall\,t< t_\ast+\delta).
\end{eqnarray*}
Since $\xi\mapsto x(s;\xi)$ is invertible and $\det \frac{\p x}{\p \xi}(s;\xi)>0$, the maps $\psi_1$ and $\psi_1(t,\cdot)$
are $C^1$-diffeomorphisms. 
Furthermore, there exists $\ep_2>0$ such that
\begin{eqnarray}\label{ep2}
 \{x(t_\ast;\xi)\,|\,\xi\in\interface_{\tilde{\ep}_1}(0)\}\supset
\interface_{\ep_2}(t_\ast).
\end{eqnarray}
Let $\varphi_1:  O_1\to \interface_{\tilde{\ep}_1}(0)$ be the function defined via 
$$\psi_1(t,\xi)=(t,x)\Leftrightarrow (t,\xi)=\psi^{-1}_1(t,x)=:(t,\varphi_1(t,x)),$$
where  we note that $\varphi_1(t,\interface(t))=\interface(0)$.  
 Then, the method of characteristics (see Appendix 1 in \cite{BFS}) together with \eqref{1xsxs}--\eqref{2xsxs3} guarantees that  the function
$$\phi_1: O_1 \to\R,\quad \phi_1(t,x):=\Phi(t; \varphi_1(t,x))$$
is $C^2$-smooth and satisfies
\begin{align*}
\frac{\p \phi_1}{\p t}(t,x) +v\Big(t,x-\frac{\phi_1(t,x)\nabla\phi_1(t,x)}{|\nabla\phi_1(t,x)|^2}\Big)\cdot \nabla \phi_1(t,x)=0&\quad (\forall\,(t,x)\in  O_1),\\\medskip
\phi_1(0,x)=\phi^0(x)&\quad (\forall\,x\in  \interface_{\tilde{\ep}_1}(0)),\\\medskip
\{x\,|\,\phi_1(t,x)=0\}=\interface(t)&\quad (\forall\,t\in(-\delta,t_\ast+\delta)),\\\medskip
  \nabla \phi_1(t,x)=p(t;\varphi_1(t,x))&\quad  (\forall\,(t,x)\in  O_1),\\\medskip
  |\nabla \phi_1(t,x)|=|p(t;\varphi_1(t,x))|=|\nabla\phi^0(\varphi_1(t,x))|& \quad (\forall\,(t,x)\in  O_1).
\end{align*}
\indent {\bf Step 4.} {\it Iteration of  Step 1-3 with the common constant $t_\ast$ given in \eqref{t-star}}.   
We investigate the PDE in \eqref{S-HJ2} for $t\ge t_\ast$ with initial  condition $\phi(t_\ast,\cdot)=\phi_1(t_\ast,\cdot)$ on $\interface_{\ep_2}(t_\ast)$ by means of the same reasoning as Step 1--3.  
The essential fact is that $|\nabla \phi_1(t_\ast,\cdot)|$ on $\interface(t_\ast)$ has exactly the same upper/lower bound as $|\nabla \phi^0|$ on $\interface(0)$,  and hence the same constants $U_1, \ldots,U_7$ and $t_\ast$ are available.  
Then, there exists $\tilde{\ep}_2\in(0,\ep_2]$  and $0<\delta\ll1$  for which  we obtain:   
\begin{align*}
&O_2 :=\bigcup_{t_\ast-\delta< s< 2t_\ast+\delta} \Big(\{s\}\times \{x(s;\xi)\,|\,\xi\in\interface_{\tilde{\ep}_2}(t_\ast)\}\Big),\\
&\psi_2: (t_\ast-\delta,2t_\ast+\delta)\times \interface_{\tilde{\ep}_2}(t_\ast)\to  O_2,\quad \psi_2(s,\xi):=(s, x(s;\xi)),\\
&\mbox{there exist $\ep_3>0$ such that
$ \{x(2t_\ast;\xi)\,|\,\xi\in\interface_{\tilde{\ep}_2}(t_\ast)\}\supset \interface_{\ep_3}(2t_\ast)$},\\
&\psi_2(t,\xi)=(t,x)\Leftrightarrow (t,\xi)=\psi_2^{-1}(t,x)=:(t,\varphi_2(t,x)),\\
&\phi_2: O_2 \to\R,\quad \phi_2(t,x):=\Phi(t; \varphi_2(t,x)).
\end{align*}
The function $\phi_2$ is $C^2$-smooth and satisfies
\begin{align*}
\frac{\p \phi_2}{\p t}(t,x) +v\Big(t,x-\frac{\phi_2(t,x)\nabla\phi_2(t,x)}{|\nabla\phi_2(t,x)|^2}\Big)\cdot \nabla \phi_2(t,x)=0&\quad  (\forall\,(t,x)\in  O_2),\\
\phi_2(t_\ast,x)=\phi_1(t_\ast,x)&\quad ( \forall\,x\in  \interface_{\tilde{\ep}_2}(0)),\\
\{x\,|\,\phi_2(t,x)=0\}=\interface(t)&\quad (\forall\,t\in(t_\ast-\delta,2t_\ast+\delta)),\\
  \nabla \phi_2(t,x)=p(t;\varphi_2(t,x))&\quad  (\forall\,(t,x)\in  O_2),\\
|\nabla \phi_2(t,x)|=|\nabla\phi_1(t_\ast,\varphi_2(t,x))|
=|\nabla\phi^0(\varphi_1(t_\ast,\varphi_2(t,x)))| &\quad (\forall\,(t,x)\in  O_2).
\end{align*}
With the common constant  $t_\ast$ and (shrinking) constants $\ep_1,\tilde{\ep}_1,  \ep_2,\tilde{\ep}_2,\ep_3, \tilde{\ep}_3,  \cdots$,  this process is repeatable countable many times.  The proof is done  with $\Theta:= \cup_{l\in\N}  O_l$.
\subsection{Proof of Theorem \ref{Thm-viscosity}}
Our proof is based on the standard theory of viscosity solutions, i.e., Perron's method for existence, and the comparison principle for uniqueness and continuity: see  \cite{Ishii} and \cite{CIL}. 
  
We first recall the definition of viscosity (sub/super)solutions of a general first-order Hamilton-Jacobi equation
\begin{eqnarray}\label{HJ2}
\mathcal{H}(z,u(z),\nabla_z u(z))=0 \mbox{\quad in $O$},
\end{eqnarray}
where $O\subset\R^N$ is an open set, $\mathcal{H}=\mathcal{H}(z,u,q):O\times\R\times\R^N\to\R$ is a given continuous function and $u:O\to\R$ is the unknown function.  Evolutional  Hamilton-Jacobi equations are included in this form with $z=(t,x)$. 
For a locally bounded function $u:O\to\R$, let  $u^\ast:O\to\R$ be  the upper semicontinuous envelope   of $u$ and $u_\ast:O\to\R$ be the lower semicontinuous envelope of $u$, i.e., 
\begin{eqnarray*}
&&u^\ast(z):=\lim_{r\to0}\,\sup\{ u(y)\,|\,y\in O,\,\,\,0\le|y-z|\le r\},\\
&&u_\ast(z):=\lim_{r\to0}\,\inf\{ u(y)\,|\,y\in O,\,\,\,0\le|y-z|\le r\},
\end{eqnarray*}
where $u^\ast$ is upper semicontinuous and $u_\ast$ is lower semicontinuous;  if $u$ is upper semicontinuous (resp. lower semicontinuous),  $u=u^\ast$ (resp. $u=u_\ast$) holds. 
 \medskip

\noindent {\bf Definition.} {\it  A function $u:O\to\R$ is a viscosity subsolution (resp.\,\,supersolution) of \eqref{HJ2}, provided
\begin{itemize}
\item $u^\ast$ is bounded from above (resp.\,\,$u_\ast$ is bounded from below);
\item If $(\varphi,z)\in C^1(O;\R)\times O$ is such that 
\begin{align*}
&\mbox{$u^\ast-\varphi$ has a local maximum at the point $z$},\\
&\mbox{(resp.\,\,$u_\ast-\varphi$ has a local minimum at the point $z$)}
\end{align*}
 it holds that 
 $$\mathcal{H}(z,u^\ast(z),\nabla_y \varphi(z))\le0
 \quad(\mbox{resp.\,\,} \mathcal{H}(z,u_\ast(z),\nabla_y \varphi(z))\ge0)
 .$$
\end{itemize}
A function $u:O\to\R$ is a viscosity solution of \eqref{HJ2}, if it is both a  viscosity subsolution and a viscosity supersolution of  \eqref{HJ2}.
}
\medskip\medskip

In Perron's method, construction of an appropriate viscosity subsolution and supersolution of \eqref{S-HJm}, whose  zero level sets coincide with $\{\interface(t)\}_{t\ge0}$, is essential; the linear transport equation with an external force provides them. 
First, we state a lemma. 
\begin{Lemma}\label{v-tra}
Suppose that $v$ satisfy  (H1) and (H2). Let $\phi^0:\bar{\Omega}\to\R$ be continuous.  
Then, $f(t,x):=\phi^0(X(0,t,x))$ satisfies $\frac{\p f}{\p t}(t,x)+v(t,x)\cdot \nabla f(t,x)=0$ on $(0,T)\times\Omega$ in the sense of viscosity solutions. 
\end{Lemma}
\begin{proof}
See the proof of Lemma 3.2 in \cite{BFS}.
\end{proof}
Now we recall that our velocity feild $v$ is extended over $[0,\infty)\times\R^3$ being continuous in $(t,x)$ and Lipschitz continuous in $x$.
 Let $V_0>0$ be a constant such that 
\begin{align}\label{VVV}
|v(t,x)-v(t,y)|\le V_0|x-y|\quad (\forall\,t\ge0, \,\,\forall\,x,y\in\R^3).
\end{align}
Then it holds that 
\begin{align}\label{4esjo}
-V_0 |u|\le v\Big(t,x-\frac{\theta^2up}{\theta^2|p|^2+\eta(\theta^2|p|^2)}\Big)\cdot p -  v(t,x)\cdot p\le V_0 |u|\quad (\forall\,(t,x,p,u),\,\,\,\forall\,\theta\ge0).
\end{align}
In fact, we have  
\begin{align*}
&\Big|v\Big(t,x-\frac{\theta^2up}{\theta^2|p|^2+\eta(\theta^2|p|^2)}\Big)\cdot p -  v(t,x)\cdot p\Big|
=V_0|u| \frac{\theta^2 |p|^2}{\theta^2|p|^2+\eta(\theta^2|p|^2)}  \le V_0|u|.  
\end{align*}
In this subsection, we use \eqref{4esjo} with $\theta=1$.  Define the function $S:[0,\infty)\times\bar{\Omega}\to\R$ as 
\begin{align*}
&S(t,x):=\left\{
\begin{array}{cll}
&-V_0&\mbox{for $ x\in\overline{\Omega^+(t)}$, $t\in[0,\infty)$},\\
&+V_0&\mbox{for $ x\in\Omega^-(t)$, $t\in[0,\infty)$}.
\end{array}
\right.
\end{align*}
Introduce  the functions $\rho,\tilde{\rho}:[0,T)\times\bar{\Omega}\to\R$ as
\begin{align*}
&\rho(t,x):=\phi^0(X(0,t,x))\exp\Big(\int_0^t S(s,X(s,t,x)){\rm d}s\Big)\quad (\mbox{$X$ is the flow of \eqref{ODE}}),\\
&\tilde{\rho}(t,x):=\phi^0(X(0,t,x))\exp\Big(\int_0^t -S(s,X(s,t,x)){\rm d}s\Big), 
\end{align*}
which are  continuous (hence, $\rho^\ast=\rho=\rho_\ast$, $\tilde{\rho}^\ast=\tilde{\rho}=\tilde{\rho}_\ast$).  
We show that $\rho$ (resp.\,$\tilde{\rho}$) is a viscosity subsolution (resp.\,supersolution) of \eqref{S-HJm}. 
We deal with $\tilde{\rho}$ to be a visocosity supersolution.  
Fix any $(t,x)\in (0,\infty)\times\Omega$.  
Let $\varphi$ be any test function satisfying the condition of the test for viscosity supersolutions at $(t,x)$:
\begin{align}\label{321cccc}  
\tilde{\rho}(s,y)-\varphi(s,y)
\ge \tilde{\rho}(t,x)-\varphi(t,x)
\mbox{\quad ($\forall\, (s,y)$ near $(t,x)$),}
\end{align}
\indent {\bf Case 1:} $(t,x)\in \cup_{0< s<\infty}(\{s\}\times\interface(s))$. 
Then, $\tilde{\rho}(s,X(s,t,x))=\tilde{\rho}(t,x)= 0$ for all $0\le s\le t$.
Hence,  taking $(s,y)=(s,X(s,t,x))$ in \eqref{321cccc} and sending $s\to t-$, we find  
 \begin{eqnarray*}
 {\rm D}\varphi(t,x)\cdot (1,v(t,x))\ge0\quad ({\rm D}:=(\p_s,\nabla_y)), 
 \end{eqnarray*}
which yields 
$$ \frac{\p \varphi}{\p s}(t,x)+v(t,x)\cdot\nabla_y \varphi(t,x)\ge 0.$$
Therefore, noting that $\tilde{\rho}(t,x)=0$, we obtain the desired inequality: 
$$\frac{\p \varphi}{\p s}(t,x)+v\Big(t,x-\frac{\tilde{\rho}(t,x)\nabla_y \varphi(t,x)}{|\nabla_y \varphi(t,x)|^2+\eta(|\nabla_y \varphi(t,x)|^2)}\Big)\cdot\nabla_y \varphi (t,x)\ge 0.$$
\indent {\bf Case 2:} $(t,x)\not\in \cup_{0< s<T}(\{s\}\times\interface(s))$. 
Setting $r:=\tilde{\rho}(t,x)-\varphi(t,x)$ in \eqref{321cccc},  we have 
\begin{align}\label{31case11}  
\tilde{\rho}(s,y)-(\varphi(s,y)+r)
\ge 0= \tilde{\rho}(t,x)-(\varphi(t,x)+r) \mbox{\quad ($\forall\, (s,y)$ near $(t,x)$).}
\end{align}
Since $\exp(\int_0^s S(s',X(s',s,y)){\rm d}s')>0$, multiplying \eqref{31case11} by  $\exp(\int_0^s S(s',X(s',s,y)){\rm d}s')$  yields  
\begin{align*}  
&\phi^0(X(0,s,y))-(\varphi(s,y)+r)\exp\Big(\int_0^s S(s',X(s',s,y)){\rm d}s'\Big)
\ge 0,
\end{align*}
while 
\begin{align*}
0&= \tilde{\rho}(t,x)-(\varphi(t,x)+r)\\
&= \tilde{\rho}(t,x)\exp\Big(\int_0^t S(s',X(s',t,x)){\rm d}s'\Big)-(\varphi(t,x)+r)\exp\Big(\int_0^t S(s',X(s',t,x)){\rm d}s'\Big)\\
&= \phi^0(X(0,t,x))-(\varphi(t,x)+r)\exp\Big(\int_0^t S(s',X(s',t,x)){\rm d}s'\Big).
\end{align*} 
Hence, setting $\psi(s,y):=(\varphi(s,y)+r)\exp(\int_0^s S(s',X(s',s,y)){\rm d}s')$,  we obtain 
\begin{align*}  
&\phi^0(X(0,s,y))-\psi(s,y)
\ge  \phi^0(X(0,t,x))-\psi (t,x) \mbox{\quad ($\forall\, (s,y)$ near $(t,x)$).}
\end{align*}
Since $\psi(s,y)$ is $C^1$-smooth near $(t,x)$, $\psi$ serves as a test function for $f(s,y):=\phi^0(X(0,s,y))$ at $(t,x)$; since Lemma \ref{v-tra} confirms that $f$ is a viscosity solution of $\frac{\p f}{\p t}+v\cdot \nabla f=0$ on $(0,T)\times\Omega$, we  obtain 
\begin{align}\label{33vis3333}
\frac{\p \psi}{\p s}(t,x)+v(t,x)\cdot\nabla_y\psi(t,x)\ge 0.
\end{align}
Set  $g(s,y):=\int_0^s S(s',X(s',s,y)){\rm d}s'$. Then, $\psi=(\varphi+r) e^g$ and  the left-hand side of \eqref{33vis3333} is   
\begin{align*}
\Big(\frac{\p \varphi}{\p s}(t,x)+v(t,x)\cdot\nabla_y\varphi(t,x)\Big) e^{g(t,x)}+\Big( \frac{\p g}{\p s}(t,x)+v(t,x)\cdot \nabla_y g(t,x) \Big)(\varphi(t,x)+r)e^{g(t,x)}.
\end{align*}
 A direct calculation yields $\frac{\p g}{\p s}(t,x)+v(t,x)\cdot \nabla_y g(t,x)=S(t,x)$. 
Therefore, we obtain 
$$\frac{\p \varphi}{\p s}(t,x)+v(t,x)\cdot\nabla_y\varphi(t,x)\ge -(\varphi(t,x)+r)S(t,x)=-\tilde{\rho}(t,x)S(t,x),$$
which is equivalent to 
\begin{align*}
&\frac{\p \varphi}{\p s}(t,x)
+v\Big(t,x-\frac{\tilde{\rho}(t,x)\nabla_y \varphi(t,x)}{|\nabla_y \varphi(t,x)|^2+\eta(|\nabla_y \varphi(t,x)|^2)}\Big)\cdot\nabla_y \varphi (t,x)\\
&\ge 
v\Big(t,x-\frac{\tilde{\rho}(t,x)\nabla_y \varphi(t,x)}{|\nabla_y \varphi(t,x)|^2+\eta(|\nabla_y \varphi(t,x)|^2)}\Big)\cdot\nabla_y \varphi (t,x)
-v(t,x)\cdot\nabla_y\varphi(t,x)\\
&\qquad -\tilde{\rho}(t,x)S(t,x)\\
&=:q-\tilde{\rho}(t,x)S(t,x).
\end{align*}
If $x\in \Omega^+(t)$,  we have $-\tilde{\rho}(t,x)<0$ and the  inequality  \eqref{4esjo} provides  $q\ge -\tilde{\rho}(t,x)V_0=\tilde{\rho}(t,x)S(t,x)$; if $x\in \Omega^-(t)$, we have $-\tilde{\rho}(t,x)>0$ and  $q\ge -\tilde{\rho}(t,x)(-V_0)=\tilde{\rho}(t,x)S(t,x)$ due to \eqref{4esjo}.   
Therefore, we obtain the desired inequality: 
\begin{align*}
\frac{\p \varphi}{\p s}(t,x)
+v\Big(t,x-\frac{\tilde{\rho}(t,x)\nabla_y \varphi(t,x)}{|\nabla_y \varphi(t,x)|^2+\eta(|\nabla_y \varphi(t,x)|^2)}\Big)\cdot\nabla_y \varphi (t,x)\ge 0,
\end{align*}
which concludes that $\tilde{\rho}$ is a viscosity supersolution. A similar reasoning shows that $\rho$ is a viscosity subsolution. 

 We apply Perron's method with the viscosity subsolution $\rho$ and supersolution $\tilde{\rho}$.  
It holds that 
\begin{align*}
&\rho ,\tilde{\rho}\in C^0([0,T)\times\bar{\Omega}),\quad
\rho\le\tilde{\rho}\mbox{ in $[0,\infty)\times\bar{\Omega}$},\quad \rho=\tilde{\rho}=0\mbox{ on $[0,\infty)\times \p\Omega$},\\
&\rho(0,\cdot)=\tilde{\rho}(0,\cdot)=\phi^0 \mbox{ on }\bar{\Omega},\quad  \rho=\tilde{\rho}=0\mbox{ on }\bigcup_{0\le s<\infty}\Big(\{s\}\times\interface(s)\Big).
\end{align*}
It follows from Perron's method (Theorem 3.1 in \cite{Ishii}) that there exists a viscosity solution $\phi:(0,\infty)\times\Omega\to\R$ of the PDE in \eqref{S-HJm}  such that
\begin{align}\label{ppee}
\rho\le \phi\le \tilde{\rho} \mbox{ in $(0,\infty)\times\Omega$}.
\end{align}
At this point, the continuity of $\phi$ in $(0,\infty)\times\Omega$ is not clear, but \eqref{ppee} implies that $\phi$ can be continuously extended at least to the boundary $([0,\infty)\times\p\Omega)\cup(\{0\}\times\bar{\Omega})$ as  \begin{align}\label{boundary}
\phi=0\mbox{ on $[0,\infty)\times\p\Omega$},\quad \phi(0,\cdot)=\phi^0\mbox{ on $\{0\}\times\bar{\Omega}$}, 
\end{align}
where we note that  \eqref{boundary} is the desired initial/boundary condition. 
Furthermore, $\phi$ satisfies \eqref{visin}.

We apply  the comparison principle (Theorem 8.2 in \cite{CIL}) to confirm the continuity of $\phi$ on $[0,\infty)\times\bar{\Omega}$ and the uniqueness of a viscosity solution.   
For this purpose, the monotonicity condition must be verified through the change of variable $w=e^{-V_0t}\phi$, i.e.,   the problem \eqref{S-HJm} is changed into  
\begin{align}\label{HJmono}
&\frac{\p w}{\p t}(t,x)+v\Big(t,x-\frac{e^{2V_0t}w(t,x)\nabla w(t,x)}{e^{2V_0t}|\nabla w(t,x)|^2+\eta(e^{2V_0t}|\nabla w(t,x)|^2)}\Big)\cdot \nabla w(t,x)  +V_0 w(t,x)=0,
\end{align}
where the function 
\begin{align*}
&u\mapsto G(t,x,p,u):=v\Big(t,x-\frac{e^{2V_0t}up}{e^{2V_0t}|p|^2+\eta(e^{2V_0t}|p|^2)}\Big)\cdot p+V_0u   
\end{align*}
 is nondecreasing for each $(t,x,p)$. In fact, due to \eqref{VVV} and $\eta(\cdot)\ge0$,  we have for every $\ep\ge0$,  
\begin{align*}
&G(t,x,p,u+\ep)-G(t,x,p,u)\\
&\quad = V_0\ep+
v\Big(t,x-\frac{e^{2V_0t}(u+\ep)p}{e^{2V_0t}|p|^2+\eta(e^{2V_0t}|p|^2)}\Big)\cdot p-v\Big(t,x-\frac{e^{2V_0t}up}{e^{2V_0t}|p|^2+\eta(e^{2V_0t}|p|^2)}\Big)\cdot p\\
&\quad \ge  V_0\ep- V_0\ep\frac{e^{2V_0t} |p|^2}{e^{2V_0t}|p|^2+\eta(e^{2V_0t}|p|^2)}\ge 0.
\end{align*}
The new function $w$ is still a viscosity solution of \eqref{HJmono} (see Subsection 3.2 of \cite{BFS} for a  related explanation/calculation) satisfying the initial/boundary condition $w(0,\cdot)=\phi^0$, $w|_{\p\Omega}=0$ together with the equality $w_\ast=w^\ast=w$ on $([0,\infty)\times\p\Omega)\cup (\{0\}\times\bar{\Omega})$.   Hence, it is enough to show the continuity and uniqueness of $w$. 
By definition,  $w_\ast\le w^\ast$ holds. On the other hand, $w$ being both a viscosity subsolution and viscosity supersolution implies that $w^\ast$ is an upper semicontinuous viscosity subsolution and $w_\ast$ is a lower semicontinuous viscosity supersolution; the comparison principle implies that $w^\ast\le w_\ast$ on $[0,\infty)\times\Omega$.
 Therefore, we  conclude that $w$ is continuous on $[0,\infty)\times\bar{\Omega}$ satisfying the initial/boundary condition in the classical sense. Furthermore,  such a viscosity solution is unique. In fact, if $\tilde{w}$ is another (continuous) viscosity solution, the comparison principle implies $w\le\tilde{w}$ by regarding $w$ as a viscosity subsolution and $\tilde{w}$ as a viscosity supersolution; $w\ge\tilde{w}$ by regarding $w$ as a viscosity supersolution and $\tilde{w}$ as a viscosity subsolution.
 
We conclude this subsection with the  remark that the above reasoning works also for \eqref{S-HJm} with initial data $\phi^0$ that is given independently from the two-phase flow with an moving  interface, where one still needs to look at the sign of $\phi^0(X(0,t,x))$ in construction of  suitable viscosity supersolution and viscosity subersolution.

\subsection{Proof of Theorem \ref{Thm-comparison}}
Let $V_0$ be the constant in \eqref{VVV} and $T>0$ be an arbitrary fixed number.  
The PDE \eqref{HJmono} is denoted by 
\begin{align}\label{HJcom}
\frac{\p w}{\p t}(t,x)+G(t,x,\nabla w(t,x),w(t,x))=0,
\end{align}
where we recall that 
$$   u\mapsto G(t,x,p,u):=v\Big(t,x-\frac{e^{2V_0t}up}{e^{2V_0t}|p|^2+\eta(e^{2V_0t}|p|^2)}\Big)\cdot p+V_0u $$
is nondecreasing for each $(t,x,p)$.    
We state two estimates in regards to $G$, which play an essential role in the upcoming argument. 
We write  
\begin{align*}
&G(t,x,p+q,u)-G(t,x,p,u)=\uwave{v\Big(t,x-\frac{e^{2V_0t}u(p+q)}{e^{2V_0t}|p+q|^2+\eta(e^{2V_0t}|p+q|^2)}   \Big)\cdot q}_{\rm (i)}\\
&\qquad +\uwave{\Big\{v\Big(t,x-\frac{e^{2V_0t}u(p+q)}{e^{2V_0t}|p+q|^2+\eta(e^{2V_0t}|p+q|^2)}   \Big)
-v\Big(t,x-\frac{e^{2V_0t}up}{e^{2V_0t}|p|^2+\eta(e^{2V_0t}|p|^2)}   \Big)\Big\}\cdot p}_{\rm (ii)}\!\!\!\!.
\end{align*}
 There exist  some constants $c_1=c_1(T)$ and $c_2=c_2(T)>0$ for which we have for all $(t,x,p,q,u)\in [0,T]\times\R^3\times\R^3\times\R^3\times\R$, 
\begin{align}\label{G3} 
|{\rm (i)}-v(t,x)\cdot q|&\le c_1(T)|u||q|, \\\label{G4}
|{\rm (ii)}|&\le  c_2(T)|u||p||q|.
\end{align}
In fact, we find some constant $c_3(T)>0$ such that  for each $r\ge0$ and  $t\in[0,T]$,  
\begin{align*}
&\frac{e^{2V_0t}r}{e^{2V_0t}r^2+\eta(e^{2V_0t}r^2)}= \frac{r}{r^2+e^{-2V_0t}\eta(e^{2V_0t}r^2)}\le c_3(T),\\
&|{\rm (i)}-v(t,x)\cdot q|
\le   V_0 \frac{e^{2V_0t}|u||p+q|}{e^{2V_0t}|p+q|^2+\eta(e^{2V_0t}|p+q|^2)}|q|
\le V_0c_3(T)|u||q|,
\end{align*}
which gives \eqref{G3}. 
%
As for \eqref{G4}, setting $R(t,p):=\frac{e^{2V_0t}p}{e^{2V_0t}|p|^2+\eta(e^{2V_0t}|p|^2)}$, we have 
\begin{align*}
|{\rm (ii)}|&=\Big| \Big\{v\Big(t,x-uR(t,p+q)\Big)-v\Big(t,x-uR(t,p)\Big)\Big\} \cdot p \Big|\\
&\le V_0|u||p| |R(t,p+q)-R(t,p)|. 
\end{align*}
Since $R$ is $C^1$-smooth, we see that there exist some $\theta_{i}\in(0,1)$ such that 
\begin{align*}
|R(t,p+q)-R(t,p)|=\sqrt{\sum_{i=1}^3\Big\{\frac{\p R_i}{\p p}(t,p+\theta_i q)\cdot q \Big\}^2}
\le |q|\sqrt{\sum_{i=1}^3\Big|\frac{\p R_i}{\p p}(t,p+\theta_i q)\Big|^2}.
\end{align*}
There exists some constant $c_4(T)>0$ such that for all $(t,p)\in[0,T]\times\R^3$  
\begin{align*}
\Big|\frac{\p R_i}{\p p_j}(t,p)\Big|
&= \Big|\frac{e^{2V_0t}\delta_{ij}}{e^{2V_0t}|p|^2+\eta(e^{2V_0t}|p|^2)} 
- \frac{2\{1+\eta'(e^{2V_0t}|p|^2)\}e^{2V_0t}e^{2V_0t}p_ip_j}{\{e^{2V_0t}|p|^2+\eta(e^{2V_0t}|p|^2)\}^2}\Big|\\
&\le c_4(T) +2(1+\max |\eta'|)   \Big\{\frac{e^{2V_0t}|p|}{e^{2V_0t}|p|^2+\eta(e^{2V_0t}|p|^2)}\Big\}^2  \\
&\le c_4(T) + 2(1+\max |\eta'|)c_3(T)^2,
\end{align*}
which leads to \eqref{G4}. 

For the smooth solution $\phi:\Theta\to\R$ of Theorem \ref{main-thm2}, the function $w(t,x):=e^{-V_0 t}\phi$ satisfies \eqref{HJcom} in the classical sense, where we recall \eqref{eta2}.  
Let $\tilde{\phi}:[0,\infty)\times\bar{\Omega}\to\R$ denote the viscosity solution in Theorem \ref{Thm-viscosity}. 
The function $\tilde{w}(t,x):=e^{-V_0 t}\tilde{\phi}$ satisfies \eqref{HJcom} in the sense of viscosity solutions.

We (partly) describe $\Theta_T:=\Theta|_{0\le t \le T}$  as a family of smooth cylinders:
\begin{itemize}
\item With $\beta(T):=\sup_{\Theta_T}|\nabla w(t,x)|=\sup_{\Theta_T}|e^{-V_0t}\nabla \phi(t,x) |$, take  a constant  $\alpha>0$  such that
\begin{align}\label{3alpha}
\alpha -2V_0 -c_1(T)\beta(T)
-c_2(T)\beta(T)^2>0 ;
\end{align}
\item Consider the function
$$\bar{u}(t,x):= e^{\alpha t}w(t,x) ,$$
where $\bar{u}$ solves the following equation in the classical sense: 
\begin{align}\label{3pen}
 \frac{\p u}{\p t}(t,x)+e^{\alpha t}G(t,x,e^{-\alpha t}\nabla u(t,x),e^{-\alpha t}u(t,x))=\alpha u(t,x)
 \quad u(0,\cdot)=\phi^0;
\end{align}
\item With some constant $\bar{m}>0$, define
\begin{align*}
&A_m:=\{ (t,x)\in[0,T]\times\Omega \,|\,\bar{u}(t,x)=m\} \mbox{\quad for each $ -\bar{m}\le m\le \bar{m}$},  \\
&\Gamma_T:= \bigcup_{-\bar{m}\le m\le \bar{m}} A_m.
\end{align*}
\end{itemize}
Since $\nabla \bar{u}\neq0$ near $\cup_{0\le t\le T} (\{t\}\times\interface(t))$, we may choose  $\bar{m}>0$ (possibly very small) so that  each $A_m$ is a smooth surface,  $\Gamma_T$ is contained in $\Theta_T$ and $\p\Gamma_T\setminus \Gamma_T|_{t=0,T}=(A_{\bar{m}}\cup A_{-\bar{m}})|_{0<t<T}$, while $\Gamma_T$ contains the $[0,T]$-part of  a tubular neighborhood of  $\{\interface(t)\}_{t\ge0}$. Note that such an $\bar{m}$ depends, in general, on $T$.

We will prove that $\phi\equiv \tilde{\phi}$ on $\Gamma_T$.    Suppose that this is not the case. 
Then, one of the following inequality holds: 
\begin{align*}
\max_{\Gamma_T}(\tilde{\phi}-\phi)>0,\quad \max_{\Gamma_T}(\phi-\tilde{\phi})>0.
\end{align*}
We deal with the first case (the other case is parallel). 
Then, there exists an interior point  $(t^\ast,x^\ast)$ of $\Gamma_T$ such that
$$\sigma:=\tilde{w}(t^\ast,x^\ast)-w(t^\ast,x^\ast)>0.$$
Since the  zero level set  of $w$ and that of $\tilde{w}$ are identical, it holds that 
$$(t^\ast,x^\ast)\not\in A_0=\bigcup_{0\le t\le T}(\{t\}\times\interface(t)).$$
Consider the numbers $m^\ast\in(-\bar{m},\bar{m})$, $\delta>0$ and $M>0$ such that 
\begin{align*}
&(t^\ast,x^\ast)\in A_{m^\ast},\,\,\,\mbox{or equivalently, }\bar{u}(t^\ast,x^\ast)=m^\ast,\\
&(m^\ast)^2+2\delta<(\bar{m})^2,\quad
M\ge\max_{(t,x)\in \Gamma_T} |w(t,x)-\tilde{w}(t,x)|.
\end{align*}
Take a monotone increasing $C^1$-function $h: \R\to [0,3M]$ such that
\begin{align*}
h(r)
=\left\{
\begin{array}{cll}
 &3M, &\mbox{ if  $(m^\ast)^2 + 2\delta \le r$},\\
 &0,&\mbox{ if $ r\le (m^\ast)^2+\delta$},\\
 &\mbox{monotone transition between $0$ and $3M$},& \mbox{ otherwise}.
\end{array}
\right.
\end{align*}
For each $\lambda>0$, consider the function $F_\lambda:\Gamma_T\to\R$ defined as
$$F_\lambda(t,x):= \tilde{w}(t,x) -w(t,x)-\lambda t -h(\bar{u}(t,x)^2).$$
Let  $(t_0,x_0)=(t_0(\lambda),x_0(\lambda))\in \Gamma_T$ be the maximum point of $F_\lambda$, i.e.,  
\begin{align*}
 F_{\lambda}(t_0,x_0)=\max_{ \Gamma_T}F_{\lambda}.
\end{align*}
\indent We show that {\it  there exists a  sufficiently small $\lambda>0$ for which $(t_0,x_0)$ is away from the ``lateral surface'' of $\Gamma_T$, i.e., $(t_0,x_0)\not\in A_{\pm \bar{m}}$, or equivalently, $|\bar{u}(t_0,x_0)|\neq \bar{m}$, and $t_0\neq0$.} 
Since $F_{\lambda}(t_0,x_0)\ge F_{\lambda}(t^\ast,x^\ast)$ and $h(\bar{u}(t^\ast,x^\ast)^2)=h((m^\ast)^2)=0$ by definition, we may fix $\lambda>0$ small enough  so that  
\begin{align}\label{831}
&\tilde{w}(t_0,x_0)-w(t_0,x_0)\ge F_{\lambda}(t_0,x_0)
\ge F_{\lambda}(t^\ast,x^\ast)\\\nonumber 
&=\tilde{w}(t^\ast,x^\ast)-w(t^\ast,x^\ast)-\lambda t^\ast-h(\bar{u}(t^\ast,x^\ast)^2)
=\sigma  -\lambda t^\ast
\ge \frac{\sigma}{2}>0.
\end{align}
Suppose that $(t_0,x_0)$ is on the ``lateral surface'' of $\Gamma_T$, i.e., $(t_0,x_0)\in A_{\pm \bar{m}}$, or equivalently, $|\bar{u}(t_0,x_0)|=\bar{m}$.
Then, since $h(\bar{m})=3M$, we have 
\begin{align*}
F_{\lambda}(t_0,x_0)
&=\tilde{w}(t_0,x_0)-w(t_0,x_0)-\lambda t_0
- h(\bar{m}^2)
\le M-h(\bar{m}^2)=-2M <0,
\end{align*}
which contradicts to \eqref{831}. 
Since $w\equiv\tilde{w}$ on $\Gamma_T|_{t=0}$, \eqref{831} implies that $t_0\neq0$.

Fix  $\lambda>0$ as mentioned above.  {\it We carry out the test at $(t_0,x_0)$ for $\tilde{w}$  being a viscosity subsolution with the test function $\psi$ given as} 
$$\psi(t,x):=w(t,x)+\lambda t+h(\bar{u}(t,x)^2) \mbox{\quad ($w$, $\bar{u}$ and $h$ are $C^1$-smooth)}.$$  
This makes sense, because the point $(t_0,x_0)$ is either an interior point of $\Gamma_T$, or $t_0=T$ and $(T,x_0)\not\in A_{\pm \bar{m}}$ (see the lemma in Section 10.2 of \cite{Evans-book} for a remark on the case $t_0=T$). 
The above $\psi$ confirms  that  $\tilde{w}(t,x)-\psi(t,x)$ takes a maximum at $(t_0,x_0)$. 
Hence, the test for $\tilde{w}$ being a viscosity subsolutions yields  
$$\frac{\p \psi}{\p t}(t_0,x_0)+G(t_0,x_0,\nabla \psi(t_0,x_0),\tilde{w}(t_0,x_0))\le 0.$$
Since $G(t,x,p,u)$ is nondecreasing with respect to $u$, the inequality \eqref{831} implies that 
\begin{align}\label{3super2}
\frac{\p \psi}{\p t}(t_0,x_0)+G(t_0,x_0,\nabla \psi(t_0,x_0),w(t_0,x_0))\le 0.
\end{align}
On the other hand, since $w$ is a classical solution, we have 
\begin{align}\label{3sub}
\frac{\p w}{\p t}(t_0,x_0)+G(t_0,x_0,\nabla w(t_0,x_0),w(t_0,x_0))= 0.
\end{align}
By  \eqref{3super2} and \eqref{3sub}, we obtain
\begin{align}\label{33BS}
\lambda
&\le  -2h'(\bar{u}(t_0,x_0)^2)\bar{u}(t_0,x_0)\frac{\p\bar{u}}{\p t}(t_0,x_0)
+G\Big(t_0,x_0,\nabla w(t_0,x_0),w(t_0,x_0)\Big)\\\nonumber
&\quad -G\Big(t_0,x_0,  \nabla w(t_0,x_0)+   2h'(\bar{u}(t_0,x_0)^2)\bar{u}(t_0,x_0)\nabla\bar{u}(t_0,x_0),w(t_0,x_0)\Big).
\end{align}
Since $\bar{u}(t,x)=e^{\alpha t}w(t,x)$ solves \eqref{3pen}, we have  
\begin{align*}
&\frac{\p\bar{u}}{\p t}(t_0,x_0)
=\alpha \bar{u}(t_0,x_0)-V_0\bar{u}(t_0,x_0)\\
&\quad - v\Big(t_0,x_0-\frac{e^{2V_0t_0}e^{-2\alpha t_0}\bar{u}(t_0,x_0)\nabla \bar{u}(t_0,x_0)}{e^{2V_0t_0}e^{-2\alpha t_0}|\nabla \bar{u}(t_0,x_0)|^2+\eta(e^{2V_0t_0}e^{-2\alpha t_0}|\nabla \bar{u}(t_0,x_0)|^2)}  \Big)\cdot \nabla \bar{u}(t_0,x_0)\\
&= \alpha \bar{u}(t_0,x_0)-V_0\bar{u}(t_0,x_0)\\
&\quad -v\Big(t_0,x_0-\frac{e^{2V_0t_0}w(t_0,x_0)\nabla w(t_0,x_0)}{e^{2V_0t_0}|\nabla w(t_0,x_0)|^2+\eta(e^{2V_0t_0}|\nabla w(t_0,x_0)|^2)}  \Big)\cdot \nabla \bar{u}(t_0,x_0)\\
&=  \alpha \bar{u}(t_0,x_0)-V_0\bar{u}(t_0,x_0)-v(t_0,x_0)\cdot\nabla \bar{u}(t_0,x_0)\\
&\quad +\uwave{\Big\{ v(t_0,x_0)  -v\Big(t_0,x_0-\frac{\theta^2w(t_0,x_0)\nabla w(t_0,x_0)}{\theta^2|\nabla w(t_0,x_0)|^2+\eta(\theta^2|\nabla w(t_0,x_0)|^2)}  \Big)  \Big\}\cdot  e^{\alpha t_0}\nabla w(t_0,x_0) }_{r_1}
\end{align*}
and by \eqref{4esjo} with $\theta=e^{V_0t_0}$,  
\begin{align*}
|r_1|\le V_0|w(t_0,x_0)|e^{\alpha t_0}  =V_0|\bar{u}(t_0,x_0)|.
\end{align*}
Hence, by $h'(\cdot)\ge0$, we obtain 
\begin{align*}
&-2h'(\bar{u}(t_0,x_0)^2)\bar{u}(t_0,x_0)\frac{\p\bar{u}}{\p t}(t_0,x_0)
\le -2h'(\bar{u}(t_0,x_0)^2)|\bar{u}(t_0,x_0)|^2(\alpha-2V_0) \\
&\quad + 2h'(\bar{u}(t_0,x_0)^2)\bar{u}(t_0,x_0)v(t_0,x_0)\cdot\nabla \bar{u}(t_0,x_0).
\end{align*}
Setting  $q:= 2h'(\bar{u}(t_0,x_0)^2)\bar{u}(t_0,x_0)\nabla\bar{u}(t_0,x_0)$, we write  
\begin{align*}
&G\Big(t_0,x_0,\nabla w(t_0,x_0),w(t_0,x_0)\Big) -G\Big(t_0,x_0,  \nabla w(t_0,x_0)+  q,w(t_0,x_0)\Big)={\rm(1)}+{\rm(2)},\\
&{\rm(1)}:= -v\Big(t_0,x_0- \frac{e^{2V_0t_0}w(t_0,x_0)(\nabla w(t_0,x_0)+q)}{e^{2V_0t_0}|\nabla w(t_0,x_0)+q|^2+\eta(e^{2V_0t_0}|\nabla w(t_0,x_0)+q|^2)}  \Big)  \cdot q  \\
&{\rm(2)}:= \Big\{
v\Big(t_0,x_0- \frac{e^{2V_0t_0}w(t_0,x_0)\nabla w(t_0,x_0)}{e^{2V_0t_0}|\nabla w(t_0,x_0)|^2+\eta(e^{2V_0t_0}|\nabla w(t_0,q_0)|^2)}  \Big) \\
&\qquad -v\Big(t_0,x_0-\frac{e^{2V_0t_0}w(t_0,x_0)(\nabla w(t_0,x_0)+q)}{e^{2V_0t_0}|\nabla w(t_0,x_0)+q|^2+\eta(e^{2V_0t_0}|\nabla w(t_0,x_0)+q|^2)}  \Big) 
\Big\}\cdot \nabla w(t_0,x_0).  
\end{align*}  
By \eqref{G3},  $h'(\cdot)\ge0$ and $\bar{u}(t,x)=e^{\alpha t}w(t,x)$, we obtain  
\begin{align*}
{\rm(1)}&=-v(t_0,x_0)\cdot q +[v(t_0,x_0)\cdot q+(1)]\\
&\le v(t_0,x_0)\cdot q +c_1(T)|w(t_0,x_0)||q|  \\
&=-2h'(\bar{u}(t_0,x_0)^2)\bar{u}(t_0,x_0)v(t_0,x_0)\cdot\nabla \bar{u}(t_0,x_0)\\
&\quad + c_1(T)|w(t_0,x_0)|\times 2h'(\bar{u}(t_0,x_0)^2)|\bar{u}(t_0,x_0)||\nabla\bar{u}(t_0,x_0)|\\
&=-2h'(\bar{u}(t_0,x_0)^2)\bar{u}(t_0,x_0)v(t_0,x_0)\cdot\nabla \bar{u}(t_0,x_0)\\
&\quad +  2h'(\bar{u}(t_0,x_0)^2) |\bar{u}(t_0,x_0)|^2 \times c_1(T)|\nabla w(t_0,x_0)|.  
%
\end{align*}
By \eqref{G4},  $h'(\cdot)\ge0$ and $\bar{u}(t,x)=e^{\alpha t}w(t,x)$, we obtain  
\begin{align*}
|{\rm(2)}|&\le c_2(T)|w(t_0,x_0)||\nabla w(t_0,x_0)||q|\\
&=  2h'(\bar{u}(t_0,x_0)^2)|\bar{u}(t_0,x_0)|^2\times c_2(T)|\nabla w(t_0,x_0)|^2.
\end{align*}
Applying these estimates to the right-hand side of \eqref{33BS}, where the  cancelation of the term $2h'(\bar{u}(t_0,x_0)^2)\bar{u}(t_0,x_0)v(t_0,x_0)\cdot\nabla \bar{u}(t_0,x_0)$ takes place, we find that  
\begin{align*}
0<\lambda
&\le -2h'(\bar{u}(t_0,x_0)^2)|\bar{u}(t_0,x_0)|^2\\
&\quad\times \Big\{ 
\alpha -2V_0 -c_1(T)|\nabla w(t_0,x_0)|
-c_2(T)|\nabla w(t_0,x_0)|^2
\Big\} . 
\end{align*}
The choice of $\alpha$ in \eqref{3alpha} yields a contradiction. 

The case of $\max_{\Gamma_T}(\phi-\tilde{\phi})>0$ can be treated in the same way, and we conclude that $\phi\equiv \tilde{\phi}$ on $\Gamma_T$. Since $T>0$ is arbitrary, we conclude that 
$$\phi\equiv \tilde{\phi}\mbox{ on } \Gamma:=\bigcup_{T\in\N}\Gamma_T,$$
where $\Gamma$ is a tubular neighborhood of $\{\interface(t)\}_{t\ge0}$.

\setcounter{section}{3}
\setcounter{equation}{0}
\section{Conclusion}

We briefly summarize the state of the art concerning rigorous mathematical analysis of  the reinitialization methods (RIM),  the nonlinear modification methods (NMM) and  the velocity extension method (VEM).  
The major purpose of these methods is to  stabilize the gradient of the level set function on the moving interface, avoiding the norm of  the gradient to be close to $0$ or $\infty$, or to obtain the (local) signed distance function of the moving interface.    
In the literature,  the following cases have been established: 
\begin{itemize}
\item[(1)]  RIM for evolutional first-oder Hamilton-Jacobi equations  in the whole space $\R^n$: Hamamuki-Ntovoris \cite{HN}. 
\item[(2)]  NMM for evolutional first-oder Hamilton-Jacobi equations in the whole space $\R^n$: Hamamuki \cite{H}. 
\item[(3)]  NMM for the linear transport equation in a bounded domain of $\R^3$ arising from two-phase flow problems:  Bothe-Fricke-Soga \cite{BFS}. 
\item[(4)]  VEM for the linear transport equation in a bounded domain of $\R^3$ arising from two-phase flow problems:  the present paper. 
\end{itemize}
The investigations of (1) and (2) are done in the class of $C^0$-viscosity solutions, while those of (3) and (4) are done  in the class of $C^2$-solutions as well as in the class of $C^0$-viscosity solutions including regularity analysis of viscosity solutions.  

In (1), infinitely many times reinitialization by (a version of) the corrector equation \eqref{1corr} is described in terms of homogenization of Hamilton-Jacobi equations, where  \eqref{1corr} must be solved for a sufficiently long time; the long-time solution of  \eqref{1corr} is achieved by a scaling limit in the homogenized Hamilton-Jacobi equation, where the limit of the $C^0$-viscosity solution is indeed the signed distance function of the original interface.  The method provides the signed distance function, even from the initial data that  is not the signed distance function  of the initial interface. 

In (2), it is shown that the modified problem with initial data equal to  the signed distance function of the initial interface provides a $C^0$-viscosity solution that stays close to  the signed distance function of the interface in a small neighborhood of the interface; furthermore, the viscosity solution is differentiable on the interface with the norm of the derivative equal to $1$. 
The method can control the gradient on the interface, but cannot provide the signed distance function.   

In (3), it is proven that the modified problem provides a global-in-time/local-in-space $C^2$-solution in a tubular neighborhood of the interface  that preserves the norm of the gradient on the interface, where the initial data does not need to be the (local) signed distance function of the initial interface; furthermore, a global-in-time/global-in-space $C^0$-viscosity solution that coincides with the $C^2$-solution in a tubular neighborhood of the interface  is obtained. 
A simpler modification that does not preserve the norm of the gradient but keeps the norm within an a priori given bound in a small neighborhood of the interface is also presented.  
The nonlinear modification methods can control the gradient either strongly on  the interface or weakly in a small neighborhood of the interface, but they cannot provide the local signed distance function in general. 

In (4), i.e., in the present paper,  it is proven that the method provides a global-in-time/local-in-space $C^2$-solution in a tubular neighborhood of the interface  that preserves  the norm of the gradient not only on the interface but also in a small neighborhood of  the interface, where the initial data does not need to be the local signed distance function of the initial interface; 
 in particular, if initial data is equal to  the local signed distance function of the initial interface, the $C^2$-solution is indeed the local signed distance function; furthermore, a global-in-time/global-in-space $C^0$-viscosity solution that coincides with the $C^2$-solution in a tubular neighborhood of the interface is obtained. 
Unlike RIM,  the method cannot produce the local signed distance function from initial data that is not   the local signed distance function  of the initial interface. 

\medskip\medskip\medskip

\noindent{\bf Acknowledgement.}
Dieter Bothe is supported by the Deutsche Forschungsgemeinschaft (DFG, German
Research Foundation) -- Project-ID 265191195 -- SFB 1194.
Kohei Soga is supported by JSPS Grants-in-Aid for Scientific Research (C) \#22K03391.
\medskip\medskip\medskip

\noindent{\bf Data availability.}
Data sharing not applicable to this article as no datasets were generated or analyzed during the current study.

\medskip

\noindent{\bf Conflicts of interest statement.} The authors state that there is no conflict of interest.
\renewcommand{\theequation}{A.\arabic{equation} }
\appendix
\def\thesection{Appendix}
\section{}
\setcounter{equation}{0}

We prove Lemma \ref{2exten}. 
Define the function $\tilde{v}_i:[0,\infty)\times\R^3\to\R$ as 
$$\tilde{v}_i(t,x):=\inf_{z\in\bar{\Omega}}\{  v_i(t,z)+\lambda|x-z| \}.$$
Due to the classical result (see, e.g., Subsection 3.1.1 in \cite{EG}), it holds that $\tilde{v}_i(t,\cdot)\equiv v_i(t,\cdot)$ on $\bar{\Omega}$ and $\tilde{v}_i(t,\cdot)$ is Lipschitz continuous on $\R^3$ for each fixed $t\ge0$ with a Lipschitz constant $\lambda$.  
Then, we see that continuity of $\tilde{v}_i(\cdot,x)$ with respect to $t$ for each fixed $x\in\R^3$ implies continuity of $\tilde{v}_i$ in $(t,x)$. 
 
 We show continuity of $\tilde{v}_i(\cdot,x)$ with respect to $t$ for each fixed $x\in\R^3$.  Observe that 
\begin{align*}
\tilde{v}_i(t+\delta,x)&= \inf_{z\in\bar{\Omega}}\{  v_i(t+\delta,z)+\lambda|x-z)| \}\\
&= \inf_{z\in\bar{\Omega}}\{  v_i(t,z)+\lambda|x-z|+v_i(t+\delta,z)-v_i(t,z) \}\\
&\le  \inf_{z\in\bar{\Omega}}\Big\{  v_i(t,z)+\lambda|x-z|+\sup_{\tilde{z}\in\bar{\Omega}} |v_i(t+\delta,\tilde{z})-v_i(t,\tilde{z})| \Big\}\\
&=  \inf_{z\in\bar{\Omega}}\{  v_i(t,z)+\lambda|x-z|\} + \sup_{\tilde{z}\in\bar{\Omega}} |v_i(t+\delta,\tilde{z})-v_i(t,\tilde{z})| \\
&=\tilde{v}_i(t,x)+ \sup_{\tilde{z}\in\bar{\Omega}} |v_i(t+\delta,\tilde{z})-v_i(t,\tilde{z})|,
\end{align*}
and similarly,        
 \begin{align*}
 \tilde{v}_i(t+\delta,x)&\ge  \tilde{v}_i(t,x)- \sup_{\tilde{z}\in\bar{\Omega}} |v_i(t+\delta,\tilde{z})-v_i(t,\tilde{z})|. 
\end{align*}
Due to continuity of $v_i$ in $(t,x)$ on $[0,\infty)\times\bar{\Omega}$ with $\bar{\Omega}$ being compact, $v_i$ is uniformly continuous in $I\times\bar{\Omega}$ for each bounded closed interval $I\subset[0,\infty)$. 
Hence, it holds that
\begin{align*}
 \sup_{\tilde{z}\in\bar{\Omega}} \Big|v_i(t+\delta,\tilde{z})-v_i(t,\tilde{z})\Big|\to 0\quad \mbox{as $\delta\to0$}.
 \end{align*}
Thus, we conclude that 
$$\tilde{v}_i(t+\delta,x)-\tilde{v}_i(t,x)\to 0\quad\mbox{ as $\delta\to0$},$$
and complete the proof of Lemma \ref{2exten}. 

  
\end{document}